\documentclass[11pt,reqno]{amsart}
\usepackage{amssymb,mathrsfs,graphicx,enumerate}
\usepackage{amsmath,amsfonts,amssymb,amscd,amsthm,bbm}
\usepackage{subcaption}
\usepackage{colortbl}
\usepackage{mathtools}
\usepackage{cite}
\usepackage{hyperref}
\usepackage{upgreek}
\hypersetup{hyperindex=true}
\hypersetup{hidelinks}
\usepackage{makecell}
\allowdisplaybreaks
\graphicspath{{./figure/}}

\makeatletter
\@namedef{subjclassname@2020}{\textup{2020} Mathematics Subject Classification}
\makeatother

\title[Uniform stability of the KCS-type model ]{Uniform-in-time stability and mean-field limit of the Cucker--Smale-type model with noncompact support}
\author[Ha]{Seung-Yeal Ha}
\address[Seung-Yeal Ha]{\newline Department of Mathematical Sciences and Research Institute of Mathematics\newline Seoul National University, Seoul 08826, Republic of Korea}
\email{syha@snu.ac.kr}

\author[Wang]{Xinyu Wang}
\address[Xinyu Wang]{\newline Department of Mathematical Sciences \newline Seoul National University, Seoul 08826, Republic of Korea \newline School of Mathematics \newline Harbin Institute of Technology, Harbin  150001, People's Republic of China}
\email{wangxinyu97@snu.ac.kr}

\author[Yoon]{Wook Yoon}
\address[Wook Yoon]{\newline Department of Mathematics \newline Technical University of Munich, Garching bei M\"unchen 85748, Germany}
\email{ynwk178@gmail.com}

\newtheorem{theorem}{Theorem}[section]
\newtheorem{lemma}{Lemma}[section]

\newtheorem{proposition}{Proposition}[section]
\newtheorem{remark}{Remark}[section]

\newtheorem{definition}{Definition}[section]

\newcommand{\bbr}{\mathbb R}

\newcommand{\bbz}{\mathbb Z}
\newcommand{\e}{\varepsilon}

\newcommand{\bbn} {\mathbb N}

\newcommand{\bw}{\mbox{\boldmath $w$}}

\newcommand{\cC}{\mathcal C}
\newcommand{\cD}{\mathcal D}
\newcommand{\cI}{\mathcal I}

\newcommand{\cP}{\mathcal P}

\usepackage{verbatim}

\begin{document}



\thanks{\textbf{Acknowledgment.} 
The work of S.-Y. Ha was supported by National Research Foundation (NRF) grant funded by the Korea government (MIST) (RS-2025-00514472). The work of X. Wang was supported by the Natural Science Foundation of China (grants 123B2003), the China Postdoctoral Science Foundation (grants 2025M774290), and Heilongjiang Province Postdoctoral Funding (grants  LBH-Z24167). All authors contributed equally to this work. }

\begin{abstract}
We study the uniform-in-time stability and mean-field limit of an infinite Cucker--Smale-type (ICS-type) model in a fully noncompact setting. First, we clarify the main difficulties arising from the original ICS model and its corresponding kinetic Cucker--Smale (KCS) model, when the spatial support is noncompact. In particular, we show that the velocity diameter may remain constant in time, which invalidates the classical approach based on position and velocity diameters. Moreover, we construct a counterexample showing that the original KCS model does not satisfy uniform-in-time stability in the noncompact spatial setting. To overcome these obstacles, we introduce a new ICS-type model whose communication weight depends on a pairwise spatial moment. Under suitable assumptions, we also establish the uniform-in-time stability of the ICS-type model with respect to initial data in the noncompact setting. Finally, we derive the uniform-in-time mean-field limit for the ICS-type model, together with uniform-in-time stability for the corresponding KCS-type model.
\end{abstract}
\maketitle 



\section{Introduction} \label{sec:1}

\setcounter{equation}{0} 
Understanding the mechanism of collective behaviors is a central topic in the study of complex systems \cite{Al,Buck,D-M2}.  For this, researchers proposed several particle models to describe collective dynamics, to name a few, the Winfree model \cite{Winfree},  the Kuramoto model \cite{Kuramoto}, the Vicsek model \cite{VICS-typeek}, the Hegselmann--Krause model \cite{HK},  the Cucker--Smale (in short, CS) model \cite{Cucker} and so on. Among them, the CS model \cite{Cucker,cucker2} and its corresponding kinetic model play a distinguished role due to their analytical treatment to capture velocity alignment via all-to-all interactions. A vast literature has been devoted to understanding the long-time behavior of CS-type models, including flocking, stability, and mean-field limits \cite{k4, kinetic1, kinetic4, kinetic5,k2, k12, k1, kinetic3} etc. When the initial datum is compactly supported in the position and velocity variables, the global well-posedness, flocking behavior,  and stability results for the kinetic CS (KCS) model have been extensively studied in \cite{ a2, k2,w2,w4,Ha2025,w7,k6,k1,HKZ}. Moreover, CS-type dynamics and its variants  have been extensively studied in literature, e.g., collision avoiding \cite{collision1,collision2}, hierarchical leadership \cite{leader,ex5}, rooted leadership flocking \cite{switch, L12}, stochastic flocking \cite{HQZ2022,ks1,ks2,CM-2008,DFT-2010}, multi-cluster flocking \cite{bi-cluster,bi-cluster2,bi-cluster3}, time-discrete flocking \cite{DS1,DS2}, infinite particle \cite{w1,w3,J1,w5,w11,HLR}, etc. 

In contrast to the well-studied stability and mean-field theory for compactly supported initial data, in this work we consider uniform-in-time stability and mean-field theory of the KCS-type model in a fully noncompact spatial-velocity setting. The study of noncompact initial data has a long history in  Vlasov-type equations, including the Vlasov--Poisson system \cite{C-Z-2016,L2016}, the Vlasov--Maxwell system \cite{D-L1989,S2004}, and related references. In particular, for kinetic flocking models with diffusion \cite{ks1,DFT-2010}, both spatial and velocity supports generally become noncompact for any positive time due to the diffusive effect. These observations indicate that the uniform-in-time stability and mean-field theory for the noncompact setting are both mathematically important and physically relevant. 

To set up the stage, we begin with the description of the original KCS model \cite{k1}. Let $f = f(t, x, v)$ be the one-particle distribution function, or simply kinetic density, for the KCS ensemble at phase space position $(x,v)\in\bbr^d\times\bbr^d$ and time $t$. Then, the temporal-phase space dynamics of $f$ is governed by the following Cauchy problem:
\begin{equation} \label{A-0}
	\begin{cases} 
		\displaystyle \partial_tf+ \nabla_x \cdot (v f) + \nabla_v \cdot (L[f]f)=0, \quad t > 0,~~(x, v) \in {\mathbb R}^{2d},  \vspace{6pt}\\
		\displaystyle L[f](t,x,v) :=-\kappa\int_{\mathbb{R}^{2d}} \phi(|x-x_*|_p) \left(v-v_*\right)  f(t,x_*,v_*) dx_* d v_*,\vspace{6pt} \\
		\displaystyle f \Big|_{t = 0}  =f^{\mathrm{in}},
	\end{cases}
\end{equation}
where $\phi$ is the communication weight defined by 
\[\phi(r) = \frac{1}{(1+r^2)^{\beta/2}},\quad\beta > 0.\] 
As far as the authors know, even finite time stability and mean-field theory for KCS dynamics in a fully noncompact setting, where both position and velocity distributions may have unbounded support, have not yet been systematically studied. This class includes Gaussian, sub-Gaussian, and other physically relevant profiles \cite{w8}. The existing theories typically rely either on compact velocity  support \cite{w6} or on spatial compactness \cite{HKZ,k2}.

Understanding uniform-in-time stability and the mean-field limit of KCS dynamics in such noncompact settings, therefore, remains a fundamental problem.  However, in such non-compact settings, a counterexample shows that the uniform-in-time stability estimate for the KCS model does not hold because the velocity diameter may remain as a constant (see Propositions \ref{P2.3}  and \ref{P2.6}). To overcome this difficulty, we introduce a modified alignment mechanism in which the communication weight depends on a pairwise spatial moment: 
\begin{equation} 
	\begin{cases} \label{A-1}
		\displaystyle \partial_tf+ \nabla_x \cdot (v f) + \nabla_v \cdot (L[f]f)=0, \quad t > 0,~~(x, v) \in {\mathbb R}^{2d},  \vspace{6pt}\\
		\displaystyle L[f](t,x,v) :=-\kappa\int_{\mathbb{R}^{2d}} \phi( \bw[f](t) ) \left(v-v_*\right)  f_*dx_* d v_*,\vspace{6pt} \\
				\displaystyle \bw[f](t) :=\int_{\mathbb{R}^{4d}}|x-x_*|_pff_*d x d vd x_* d v_*,\vspace{6pt} \\
		\displaystyle f \Big|_{t = 0}  =f^{\mathrm{in}}.
	\end{cases}
\end{equation}
Here, we use handy notations:
\[ f = f(t,x,v), \quad f_* = f(t, x_*, v_*), \quad x = (x^1, \cdots, x^d), \quad  |x|_p:=\left(\sum_{k=1}^d |x^k|^p\right)^{\frac{1}{p}}.\] 
Note that the KCS-type model \eqref{A-1} can be viewed as a variant of the original KCS model \eqref{A-0}. One transparent difference is that we employ a new ansatz for the communication weight function $\eqref{A-1}_3$ such that it depends on a common pairwise spatial moment, rather than on pairwise distances alone.

The above modification $\eqref{A-1}_3$ is conceptually related to several existing variants of the CS-type models that employ non-metric or distribution-dependent communication protocols. In the original KCS model \eqref{A-0}, the communication weight between two individuals depends only on their mutual distance, which assumes that the influence decays monotonically with respect to spatial separation. While this metric-based rule works well when the group is relatively concentrated, it becomes inadequate in situations where the spatial support is unbounded, as the interaction may vanish for far-apart individuals and the classical flocking analysis fails. However, empirical studies on animal groups, such as starling flocks and fish schools, suggest that interaction rules are often not purely metric \cite{08PNAS}. For instance, individuals may interact with a fixed number of nearest neighbours regardless of population density, the so-called topological interaction \cite{Haskovec2013}. Moreover, Motsch and Tadmor \cite{MT,MT22} proposed a normalized alignment model where the communication weight is renormalized by the total influence from all other particles, thereby removing the artificial dependence on the total population size and making the model more realistic in non-uniform spatial distributions. More recently, density-induced consensus protocols \cite{Mi2020M3AS} have been introduced, where agents only react to sufficiently dense crowds in their vicinity. Our model \eqref{A-1} shares a common idea with these approaches that the communication weight should depend not only on pairwise distances but also on properties of the overall particle distribution. In our case, this is achieved by introducing the global pairwise spatial moment $\eqref{A-1}_3$, which measures the spatial spread of the entire group. This global coupling provides a new physical mechanism: the influence between two individuals is modulated by how dispersed the whole flock is. Such a mechanism explains the scenarios that individuals may perceive the group's overall extent, and also enables rigorous mathematical treatment for initial data with many physically important distributions like Gaussian, sub-Gaussian, and so on.

Motivated by mean-field limit and emergent dynamics of CS-type model with noncompact support \cite{w6,w9,w8,w10}, we first consider the infinite CS-type (in short, ICS-type) model with a sender network topology in which the influence of each particle is registered by the fixed infinite nonnegative weight vector $(m_j)_{j \in {\mathbb N}}$ with the unit sum 
\[\sum_{j=1}^\infty m_{j}=1.\] 
In this setting, the Cauchy problem for the ICS-type model in a metric space $(\bbr^d, |\cdot|_p)$ reads as follows:
\begin{align}
	\begin{cases} \label{A-2}
		\displaystyle \frac{dx_i}{dt} = v_i, \quad t>0,\quad i \in {\mathbb N},\vspace{6pt}\\
		\displaystyle\frac{dv_i}{dt} = \kappa\sum_{j=1}^\infty m_{j}\phi(\bw(X))(v_j-v_i),\vspace{6pt}\\
		\displaystyle\bw(X)=\sum_{i,j=1}^\infty m_im_j|x_j-x_i|_p,	\vspace{6pt}\\
		\displaystyle	(x_i,v_i) \Big|_{t = 0} =(x_i^{\mathrm{in}},v_i^{\mathrm{in}}).
	\end{cases}
\end{align}
Here, the infinite state and mixed norm are defined as follows:
\begin{align*}
	X:~\{x_i\}_{i\in\bbn}, (x_1,x_2,\cdots), \quad  \|X\|_{q,p}:= \left(\sum_{i=1}^\infty m_i|x_i|_p^q\right)^{\frac{1}{q}},\quad \|X\|_{p}:=\|X\|_{p,p}.
\end{align*}
Note that the KCS-type model \eqref{A-1} corresponds to the kinetic model which can be obtained from \eqref{A-2} via the mean-field limit. In fact, the empirical measure generated by the solution to \eqref{A-2} is a weak solution to the KCS-type model \eqref{A-1}. A central feature of our framework is that the initial spatial configuration is not assumed to be bounded, so particles are not confined to a compact spatial domain. In this setting, classical flocking and stability arguments based on the uniform contraction of the velocity diameter cannot be applied. These phenomena create fundamental obstructions in the analysis of long-time stability. Therefore, we are interested in the following questions:
\vspace{0.1cm}
\begin{quote}
``Is this KCS-type model \eqref{A-1} uniformly stable in the noncompact setting? If so, under what conditions on system parameters, and initial datum with noncompact support, can one show the emergent dynamics for \eqref{A-1}?"
\end{quote}
\vspace{0.1cm}
In this work, we answer the above questions in an affirmative manner. First, we recall the notions of Wasserstein distance and uniform(-in-time) stability for the KCS-type model.
\begin{definition} \label{D1.1}
\begin{enumerate}
\item The Wasserstein distance $W_{q,p}$ between two measures $\mu, \nu\in\mathcal{P}(\bbr^{2d})$ is defined as follows: for $p\in[1,\infty]$ and $q\in[1,\infty)$, $W_{q,p}$ is defined as
\begin{equation*}
z = (x,v), \quad	W_{q,p}(\mu,\nu):=\inf_{\pi\in\Pi(\mu,\nu)}\left(\int_{\bbr^{2d}\times\bbr^{2d}}|z-z^*|_p^q\pi(dz,dz^*)\right)^{\frac{1}{q}},\quad W_p(\mu,\nu):=W_{p,p}(\mu,\nu).
\end{equation*}
\item
The Cauchy problem \eqref{A-1} is uniformly stable with respect to initial datum if and only if for measure-valued solutions $(\mu_t, \nu_t)$ to \eqref{A-1} with the corresponding initial data $(\mu_0, \nu_0)$, there exists a positive constant $C$ independent of $t$ such that 
\[ \sup_{0 \leq t < \infty} W_p(\mu_t, \nu_t)\leq C W_{p}(\mu_0, \nu_0).\]
\end{enumerate}
\end{definition}
\noindent Note that in a compact setting, uniform-in-time stability of the KCS model has already been studied in \cite{HKZ}. Thus, the above question asks whether we can extend the compact-support theory to the noncompact setting. \newline

The main results of this paper are four-fold. First, we present an example showing that, in the classical ICS model proposed in \cite{w1}, the velocity diameter may remain invariant in time in the noncompact spatial setting (see Proposition \ref{P2.3}). Moreover, we construct a counterexample showing that the original KCS model fails to satisfy uniform-in-time stability when the spatial support is noncompact (see Proposition \ref{P2.6}). These examples show that, in noncompact settings, a strengthened communication weight is needed to establish uniform-in-time stability of KCS-type dynamics.

Second, we establish uniform-in-time stability for the ICS-type model \eqref{A-2} with the sender network in a fully noncompact regime. Under suitable assumptions on the communication kernel and the initial configuration:
\begin{align*}
	&\sum_{i=1}^\infty m_i v_i^{\mathrm{in}}=\sum_{i=1}^\infty m_i \bar{v}_i^{\mathrm{in}}=0,\quad \kappa>\frac{4}{\tilde{c}_p}\max \left\{\frac{\|V^{\mathrm{in}}\|_{p}}{\int_{2\|X^{\mathrm{in}}\|_{p}}^\infty \phi(s)ds},\frac{\|\bar{V}^{\mathrm{in}}\|_{p}}{\int_{2\|\bar{X}^{\mathrm{in}}\|_{p}}^\infty\phi(s)ds}\right\},\quad\tilde{c}_p=\frac{2^{2-p}}{p},
\end{align*} 
we derive uniform-in-time stability estimates for global ICS-type solutions with respect to perturbations of the initial data (see Theorem \ref{T2.1}): there exists a positive constant $C$, independent of $t$, such that 
\begin{equation*} \label{A-3}
\sup_{0 \leq t < \infty} \Big( \|X(t)-\bar{X}(t)\|_{p}+\|V(t)-\bar{V}(t)\|_{p} \Big) \le C \Big (\|X^{\mathrm{in}}-\bar{X}^{\mathrm{in}}\|_{p}+\|V^{\mathrm{in}}-\bar{V}^{\mathrm{in}}\|_{p} \Big ).
\end{equation*} 
Here, $(X(t), V(t))$, $(\bar{X}(t), \bar{V}(t))$ are solutions of ICS-type model \eqref{A-2} with initial data $(X^{\mathrm{in}}, V^{\mathrm{in}})$, $(\bar{X}^{\mathrm{in}}, \bar{V}^{\mathrm{in}})$, respectively. 

Third, as a consequence of uniform-in-time stability of ICS-type model, we also establish the uniform-in-time mean-field limit for the associated KCS-type model \eqref{A-1} via particle-in-cell method. More precisely, we assume the initial data satisfy 
\begin{equation*}
	\int_{\bbr^{2d}} \Big( |x|_p^p+|v|_p^p \Big)~ \mu_0(dx,dv)<\infty,\quad \int_{\bbr^{2d}}v~\mu_0(dx,dv)=0,
	\quad \|V_{\mu_0}\|_{p}< \frac{\kappa \tilde{c}_p}{4} \int_{2\|X_{\mu_0}\|_{p}}^{\infty}\phi(s)d s.
\end{equation*}
Here, \begin{align*}
\|X_{\mu_0}\|_{p}:=\left[\int_{\bbr^{2d}}|x|_p^p~ \mu_0(dx,dv)\right]^{1/p} \qquad \text{and}\qquad 	\|V_{\mu_0}\|_{p}:=\left[\int_{\bbr^{2d}}|v|_p^p~ \mu_0(dx,dv)\right]^{1/p}.
\end{align*}
Then, the solution $\mu_t$ to the KCS-type model \eqref{A-1} can be approximated by $\mu_t^n$, the measure derived by the ICS-type model \eqref{A-2}, in Wasserstein distance uniformly in time (see Theorem \ref{T2.2}):
\[\lim_{n\to\infty} \sup_{0 \leq t <\infty}W_p(\mu_t^n,\mu_t)=0.\]
\newline

Fourth, let $\mu_t$ and $\nu_t$ be global measure-valued solutions to \eqref{A-1} corresponding to initial data $(\mu_0, \nu_0)$ satisfying the following set of conditions:
	\begin{align*}
	\begin{aligned}
		&\int_{\bbr^{2d}} \Big( |x|_p^p+|v|_p^p \Big)~ \mu_0(dx,dv)<\infty,\quad \int_{\bbr^{2d}} \Big( |x|_p^p+|v|_p^p \Big)~ \nu_0(dx,dv)<\infty,\\
		&	\int_{\bbr^{2d}}v~ \mu_0(dx,dv)=\int_{\bbr^{2d}}v~ \nu_0(dx,dv)=0,\quad \kappa>\frac{4}{\tilde{c}_p}\max \left\{\frac{\|V_{\mu_0}\|_{p}}{\int_{2\|X_{\mu_0}\|_{p}}^\infty \phi(s)ds},\frac{\|V_{\nu_0}\|_{p}}{\int_{2\|X_{\nu_0}\|_{p}}^\infty\phi(s)ds}\right\}.
	\end{aligned}
\end{align*}
Then, there exists a positive constant $C$ independent of $t$ such that 
\[ \sup_{0 \leq t < \infty} W_p(\mu_t, \nu_t)\leq C W_{p}(\mu_0, \nu_0).\]
For further details, we refer to Theorem \ref{T2.3}. \newline

The rest of this paper is organized as follows. In Section~\ref{sec:2}, we first recall previous results on the emergent dynamics of the original ICS model, and then present examples illustrating main difficulties in a noncompact setting. We also state our main results for the new ICS-type model and KCS-type model. In Section~\ref{sec:3},  we establish the uniform-in-time stability of the  ICS-type model with a sender network. In Section~\ref{sec:4}, we derive the uniform-in-time mean-field limit and uniform-in-time stability for the KCS-type equation. Finally, Section \ref{sec:5} is devoted to a brief summary of the main results and a discussion of some remaining issues for future work. In appendices, we provide proofs for Lemma \ref{L2.1}, Proposition \ref{P2.4} and Proposition \ref{P3.1}.

\vspace{.3cm}
\noindent\textbf{Gallery of Notation:} For notational simplicity, we set the infinite position and velocity vectors  as follows:
	\begin{align*}
		&X:~\{x_i\}_{i\in\bbn},~~(x_1,x_2,\cdots),\quad V:~\{v_i\}_{i\in\bbn},~~ (v_1,v_2,\cdots),\\
		&x_i=(x_i^1,x_i^2,\cdots,x_i^d)\in\bbr^d,\quad v_i=(v_i^1,v_i^2,\cdots, v_i^d)\in\bbr^d, \quad i \in {\mathbb N}.
	\end{align*}
	 We use $|\cdot|_p$ to denote the $\ell^p$-norm of a vector in $\bbr^d$, and define the mixed $(q,p)$-norms using $\{m_i\}_{i=1}^\infty$: 
\begin{align*}
	|x_i|_p&:=\left(\sum_{k=1}^d |x_i^k|^p\right)^{\frac{1}{p}},\quad \|X\|_{q,p}:= \left(\sum_{i=1}^\infty m_i|x_i|_p^q\right)^{\frac{1}{q}},\quad \|X\|_{p}:=\|X\|_{p,p},\quad \|X\|_{\infty,p} :=\sup_{i\in\bbn}|x_i|_p.
	\end{align*}
	Note that since $\sum_{i=1}^\infty m_i=1$, Jensen's inequality gives
\[\|X\|_{q_1,p}\leq \|X\|_{q_2,p}\quad \text{for}\quad0<q_1\leq q_2\leq \infty. \]
We define the infinite dimensional vector space $\ell^{q,p}$ as 
\begin{equation*}
	\ell^{q,p}:=\{X=(x_1,x_2,\cdots):\|X\|_{q,p}<\infty\}.
\end{equation*}
Then, we have
\begin{equation*}
	\ell^{q_2,p}\subset \ell^{q_1,p}\quad \text{for} \quad 0<q_1\leq q_2\leq \infty.
\end{equation*}
Let $\mu\in\cP(\bbr^d\times \bbr^d)$ be a probability measure on position-velocity space. For $p,q\in[1,\infty)$, we define
\begin{equation*}
	\|X_{\mu}\|_{q,p}:=\left[\int_{\bbr^{2d}}|x|_p^q~ \mu(dx,dv)\right]^{1/q}, \qquad 	\|V_{\mu}\|_{q,p}:=\left[\int_{\bbr^{2d}}|v|_p^q~ \mu(dx,dv)\right]^{1/q}.
\end{equation*}
If $p=q$, we write
\[  \|X_{\mu}\|_p:=\|X_\mu\|_{p,p}, \quad \|V_{\mu}\|_p:=\|V_\mu\|_{p,p}. \]

\vspace{.3cm}
\noindent For two functions $f,g$, we write $f\lesssim g$ if and only if there exists a positive constant $C>0$ such that $f\leq Cg$. We also write 
\[ f\simeq g \quad \mbox{if and only if}\quad f\lesssim g \quad \mbox{and} \quad g\lesssim f. \] 
\section{Preliminaries} \label{sec:2}
\setcounter{equation}{0} 
In this section, we first recall previous results and study several properties of the ICS model \cite{w1}. We then present examples that illustrate the main difficulties in the noncompact setting: specifically, we show that the velocity diameter of the ICS model may remain constant in time, and we construct a counterexample demonstrating that the original KCS model fails to satisfy the uniform-in-time stability. Finally, we summarize our main results for the new ICS-type model.
	\subsection{The infinite Cucker--Smale model}\label{sec:2.1}
	Consider the Cauchy problem for the infinite Cucker--Smale model in a metric space $(\bbr^d, |\cdot|_p)$:
\begin{align} \label{B-1}
	\begin{cases}
		\displaystyle \frac{dx_i}{dt} = v_i, \quad t>0,\quad i\in {\mathbb N}, \vspace{6pt}\\
		\displaystyle\frac{dv_i}{dt} = \kappa\sum_{j=1}^\infty m_{ij}\phi(|x_j-x_i|_p)(v_j-v_i),\vspace{6pt}\\
		(x_i,v_i) \Big|_{t = 0} =(x_i^{\mathrm{in}},v_i^{\mathrm{in}}),
	\end{cases}
\end{align}
where $\phi$ is a communication weight function which takes the following explicit ansatz:
\[\phi(r)=\frac{1}{(1+r^2)^{\beta/2}}, \quad \beta > 0. \]
Next, we focus on the ICS model with the sender network, i.e., $m_{ij}=m_j$ for all $i,j\in\bbn$:
\begin{align}
	\begin{cases} \label{B-2}
		\displaystyle \frac{dx_i}{dt} = v_i, \quad t>0,\quad i \in {\mathbb N},\vspace{6pt}\\
		\displaystyle\frac{dv_i}{dt} = \kappa\sum_{j=1}^\infty m_{j}\phi(|x_j-x_i|_p)(v_j-v_i),\vspace{6pt}\\
		(x_i,v_i) \Big|_{t =0} =(x_i^{\mathrm{in}},v_i^{\mathrm{in}}).
	\end{cases}
\end{align}
First, we recall the global well-posedness of \eqref{B-2}.
\begin{proposition} 
\emph{\cite{w1,w3,w5}} \label{P2.1}
	For $p\geq 2$, we assume that initial data $(X^{\mathrm{in}},V^{\mathrm{in}})$ satisfies 
	\[(X^{\mathrm{in}},V^{\mathrm{in}})\in\ell^{p,p}\times\ell^{\infty,p}.\] Then there exists a unique global solution $(X(t), V(t))$ to \eqref{B-2} such that 
\[  (X(t),V(t))\in \ell^{p,p}\times\ell^{\infty,p}, \quad   \|V(t)\|_{\infty,p}\leq \|V^{\mathrm{in}}\|_{\infty,p}\quad \forall~t\geq 0. \]
\end{proposition}
\begin{remark}\label{R2.1}
	If $X^{\mathrm{in}}\in\ell^{q,p}$ instead of $X^{\mathrm{in}}\in\ell^{p,p}$, we can still have a global well-posedness, where $X(t)\in\ell^{q,p}$ for all $t\geq0$. See Lemma \ref{L2.1}. 
	\end{remark}
In what follows, we mainly focus on the situation in which the initial velocities of particles are contained in a compact set, while the initial positions are not:
\[ X^{\mathrm{in}}\notin \ell^{\infty,p},\quad V^{\mathrm{in}}\in \ell^{\infty,p}.\]
and set 
\[ {\mathcal P}_\infty :=  \sup_{0 \leq t < \infty} \sup_{i \in {\mathbb N}} |v_i|_p=\|V^{\mathrm{in}}\|_{\infty,p} \]
as the uniform upper bound of $\|V(t)\|_{\infty,p}$. Since Proposition \ref{P2.1} gives
\begin{equation*}
	|x_i(t) - x_i^{\mathrm{in}}|_p\leq \int_0^t |v_i(s)|_p ds\leq {\mathcal P}_{\infty}t, 
\end{equation*}
we have 
\begin{equation*}
	|x_i(t)|_p\leq |x_i^{\mathrm{in}}|_p+ {\mathcal P}_{\infty}t,
\end{equation*}
and therefore
\[ x_i(t)\in \bbr^d \quad \mbox{for every $t\geq 0$, $i\in\bbn$}. \]
However, since 
\begin{equation*}
	\sup_{i\in\bbn} |x_i(t)|_p\geq \sup_{i\in\bbn} |x_i^{\mathrm{in}}|_p- {\mathcal P}_{\infty}t=\infty,
\end{equation*}
we have 
\[ X(t)\notin \ell^{\infty,p} \quad \mbox{for all $t\geq 0$}. \]
This implies that the positions of particles cannot be contained in a compact set for any time $t\geq 0$. 

\subsection{Basic estimates} \label{sec:2.2}
In this subsection, we study several basic properties of the solution $(X(t), V(t))$ to \eqref{B-2}.
\begin{lemma} \label{L2.1}
Let $(X,V)$ be a global solution to  \eqref{B-2} with initial data satisfying the following summability conditions:
\[ (X^{\mathrm{in}},V^{\mathrm{in}})\in\ell^{q,p}\times \ell^{\infty,p}, \quad \mbox{for $p, q \in [1, \infty]$}. \]
Then, the following assertions hold.
\begin{enumerate}
\item
Total weighted momentum is conserved along \eqref{B-2}:
\[\sum_{i=1}^\infty m_iv_i(t)=\sum_{i=1}^\infty m_i v_i^{\mathrm{in}}, \quad t \geq 0. \]
\item
The functionals $ \|X(t)\|_{q,p}$ and $ \|V(t)\|_{q,p}$ satisfy 
\[ \|V(t)\|_{q,p} \leq \|V^{\mathrm{in}} \|_{q,p}, \quad \|X(t)\|_{q,p}\leq \|X^{\mathrm{in}}\|_{q,p}+t\|V^{\mathrm{in}}\|_{q,p}, \quad t \geq 0. \]
\item
If the initial spatial configuration $X^{\mathrm{in}}$ satisfies 
\[ \sum_{i=1}^\infty m_ie^{\alpha |x_i^{\mathrm{in}}|_p}<\infty \]
for some constant $\alpha>0$, then we have
\[ \sum_{i=1}^\infty m_i e^{\alpha |x_i(t)|_p}\leq e^{\alpha \cP_\infty t}\sum_{i=1}^\infty m_i e^{\alpha |x_i^{\mathrm{in}}|_p}, \quad t \geq 0. \]
\end{enumerate}
\end{lemma}

\begin{proof}
Since the proofs are similar to those in \cite{w8,w10}, we leave them in Appendix \ref{App-A}.
\end{proof}
\begin{remark} \label{R2.2}
As a corollary of Lemma \ref{L2.1}, we show that the model \eqref{B-1} is invariant. For this, we define new observables $(Y,U)$ where $Y= \{y_i\}_{i=1}^\infty,  U=\{u_i\}_{i=1}^\infty$:
	\begin{equation*}
	v_{0} := \sum_{i=1}^\infty m_i v_i^{\mathrm{in}},\quad u_i(t):=v_i(t) - v_0, \quad  y_i(t):= x_i(t) - v_0t.
	\end{equation*}
	Then (Y, U) satisfies 
	\begin{align*}
	\begin{cases} 
		\displaystyle \frac{dy_i}{dt} = u_i, \quad t>0,\quad i \in {\mathbb N},\vspace{3pt}\\
		\displaystyle\frac{du_i}{dt} = \kappa\sum_{j=1}^\infty m_{j}\phi(|y_j-y_i|_p)(u_j-u_i),\vspace{6pt}\\
		(y_i,u_i) \Big|_{t = 0} =(x_i^{\mathrm{in}},v_i^{\mathrm{in}}-v_0),
	\end{cases}
\end{align*}
and by Lemma \ref{L2.1}, we have
\begin{equation*}
	\sum_{i=1}^\infty m_iu_i(t) = \sum_{i=1}^\infty m_i(v_i^{\mathrm{in}}-v_0)=0.
\end{equation*}
\end{remark}
For an initial configuration $(X^{\mathrm{in}}, V^{\mathrm{in}})$, we introduce spatial and velocity diameters as follows:
\[ {\mathcal D}_{X}(0) := \sup_{i, j \in {\mathbb N}} |x_i^{\mathrm{in}} - x_j^{\mathrm{in}}|, \quad {\mathcal D}_{V}(0) := \sup_{i, j \in {\mathbb N}} |v_i^{\mathrm{in}} - v_j^{\mathrm{in}}|. \]
We first recall the result under the assumption that initial positions and velocities are contained in a compact set. 
\begin{proposition} \label{P2.2}
	\emph{\cite{HLW2025}} Suppose that initial data and coupling strength satisfy a set of conditions:
\begin{align*}
\begin{aligned}
& (X^{\mathrm{in}},V^{\mathrm{in}}),~~(\bar{X}^{\mathrm{in}},\bar{V}^{\mathrm{in}})\in \ell^{\infty,2} \times \ell^{\infty,2}, \quad \sum_{i=1}^\infty m_iv_i^{\mathrm{in}}=\sum_{i=1}^\infty m_i\bar{v}_i^{\mathrm{in}}=0, \\
& \kappa>\max \left\{\frac{\cD_V(0)}{\int_{\cD_X(0)}^\infty \phi(s)ds},\frac{\cD_{\bar{V}}(0)}{\int_{\cD_{\bar{X}}(0)}^\infty\phi(s)ds}\right\},
\end{aligned}
\end{align*}	
and let $(X(t),V(t))$, $(\bar{X}(t),\bar{V}(t))$ be global solutions to \eqref{B-2} with $p=2$. Then, there exists a positive constant $C$ only depending on initial data such that
	\begin{equation*}
	\sup_{0 \leq t < \infty} \Big( \|X(t)-\bar{X}(t)\|_{\infty,2}+\|V(t)-\bar{V}(t)\|_{\infty,2} \Big) \leq C \Big(\|X^{\mathrm{in}}-\bar{X}^{\mathrm{in}}\|_{\infty,2}+\|V^{\mathrm{in}}-\bar{V}^{\mathrm{in}}\|_{\infty,2} \Big ).
	\end{equation*}
\end{proposition}
\begin{remark}
Note that Proposition \ref{P2.2} cannot be extended to the case where the initial positions are not contained in a compact set, since we would have $\cD_X(0)=\cD_{\bar{X}}(0)=\infty$.
\end{remark}
In what follows, we assume that 
\[ X(t)\notin \ell^{\infty,p} \quad \mbox{for all $t\geq 0$}. \]
One of the main difficulties when dealing with the initial data with noncompact spatial support lies in the derivation of decay estimates for $\|V(t)\|_{\infty,p}$. In fact, it is possible to construct an initial data $(X^{\mathrm{in}}, V^{\mathrm{in}})$ such that, no matter how large the coupling strength $\kappa>0$ is, the corresponding solution $(X(t), V(t))$ satisfies
\[\|V(t)\|_{\infty,p}=\|V^{\mathrm{in}}\|_{\infty,p},\quad \forall~t\geq 0,\]
i.e., the $\ell^{\infty,p}$-norm of $V(t)$ remains constant in time as can be seen in the following proposition.
\begin{proposition}\label{P2.3}
Suppose that the initial data $(X^{\mathrm{in}}, V^{\mathrm{in}})$ satisfy the following property: there exists a subsequence $\{i_\ell\}_{\ell=1}^\infty$ such that 
	\begin{equation} \label{B-2-1}
		\lim_{\ell\to\infty}\ |v_{i_\ell}^{\mathrm{in}}|_2 = {\mathcal P}_\infty:=\ \|V^{\mathrm{in}}\|_{\infty,2}, \quad \lim\limits_{\ell\to\infty}|x_{i_\ell}^{\mathrm{in}}|_2=\infty, \quad \sum\limits_{i=1}^{\infty}m_ie^{|x_i^{\mathrm{in}}|_2}= c<\infty,
	\end{equation}
	where $c>0$ is a constant, and let $(X(t), V(t))$ be a global solution to \eqref{B-2}. Then,  the $\ell^{\infty,2}$-norm of $V(t)$ remains constant in time:
	\[\|V(t)\|_{\infty,2}=\|V^{\mathrm{in}}\|_{\infty,2},\quad \forall~t\geq 0. \]
\end{proposition}
\begin{proof}
It follows from $\eqref{B-2-1}_3$ that  for any $R>0$, we have 
	\begin{equation}	\label{B-3}
		\sum\limits_{i:|x_i^{\mathrm{in}}|_2\ge R}m_i\le c e^{-R}.
	\end{equation}
For any $i\in\bbn$, one has 
	\begin{align*}
		\frac{d}{dt}|v_i|_2
		= \frac{v_i}{|v_i|_2}\cdot \dot v_i
		&\ge - \kappa\sum_{j=1}^\infty m_j \phi\big(|x_j-x_i|_2\big)\big(|v_j|_2+|v_i|_2\big) \ge - 2\kappa {\mathcal P}_\infty\sum_{j=1}^\infty m_j \phi\big(|x_j-x_i|_2\big).
	\end{align*}
This yields
	\begin{align}\label{B-4}
		|v_i(t)|_2
		\ge |v_i^{\mathrm{in}}|_2-2 \kappa{\mathcal P}_\infty \int_0^t \sum_{j=1}^\infty m_j
		\phi\big(|x_j(s)-x_i(s)|_2\big)~ds .
	\end{align}
	Moreover, for any $i,j\in\bbn$, we have
	\begin{align*}
		|x_j(t)-x_i(t)|_2 &\geq |x_j^{\mathrm{in}}-x_i^{\mathrm{in}}|_2 -|x_j(t)-x_j^{\mathrm{in}}|_2-|x_i(t)-x_i^{\mathrm{in}}|_2 \geq |x_j^{\mathrm{in}}-x_i^{\mathrm{in}}|_2-2t {\mathcal P}_\infty.
	\end{align*}
Therefore, we have
	\begin{equation}	\label{B-5}
		\phi\big(|x_j(t)-x_i(t)|_2\big)
		\le \phi\big((|x_j^{\mathrm{in}}-x_i^{\mathrm{in}}|_2-2t {\mathcal P}_\infty)_+\big).
	\end{equation}
	Now, we fix 
	\[ t>0 \quad \mbox{and} \quad \e>0. \]
Since $\lim_{r \to \infty}\phi(r)=0,$ there exists $R>2t {\mathcal P}_\infty$ such that
	\begin{align}\label{B-6}
		2\kappa t {\mathcal P}_\infty~\phi(R-2t {\mathcal P}_\infty)<\frac{\varepsilon}{2}.\end{align}
	Moreover, we choose $i_\ell$ such that $|v_{i_\ell}^{\mathrm{in}}|_2$ is sufficiently close to $\|V^{\mathrm{in}}\|_{\infty,2}$ and $|x_{i_\ell}^{\mathrm{in}}|_2$ is sufficiently large:
	\begin{equation}\label{B-7}
		|v_{i_\ell}^{\mathrm{in}}|_2 > \|V^{\mathrm{in}}\|_{\infty,2}-\e,\quad |x_{i_\ell}^{\mathrm{in}}|_2>R,\quad \text{and} \quad 2c\kappa t {\mathcal P}_\infty ~e^{-(|x_{i_\ell}^{\mathrm{in}}|_2-R)}<\frac{\e}{2}.
	\end{equation}
	Then, we use \eqref{B-5} to derive
	\begin{align}
	\begin{aligned} \label{B-7-1}
		&\sum_{j=1}^\infty m_j \phi\big(|x_j(t)-x_{i_\ell}(t)|_2\big)\\
		&\hspace{1cm} = \sum_{j:|x_j^{\mathrm{in}}-x_{i_\ell}^{\mathrm{in}}|_2\le R} m_j\phi(|x_j(t)-x_{i_\ell}(t)|_2)~
		+ \sum_{j:|x_j^{\mathrm{in}}-x_{i_\ell}^{\mathrm{in}}|_2>R} m_j~\phi\big(|x_j(t)-x_{i_\ell}(t)|_2)\\
		&\hspace{1cm} \leq \sum_{j:|x_j^{\mathrm{in}}|\geq|x_{i_\ell}^{\mathrm{in}}|_2- R} m_j~
		+ \phi\big(R-2t {\mathcal P}_\infty \big).
	\end{aligned}
	\end{align}
	We combine \eqref{B-7-1} together with \eqref{B-3} and \eqref{B-4} to find 
	\begin{align}\label{B-8}
		\begin{aligned}
			|v_{i_\ell}(t)|_2
			\ge |v_{i_\ell}^{\mathrm{in}}|_2
			- 2c\kappa t\cP_\infty ~e^{-(|x_{i_\ell}^{\mathrm{in}}|_2-R)}
			- 2\kappa t\cP_\infty~\phi\big(R-2t\cP_\infty \big).
		\end{aligned}
	\end{align}
Again, we use \eqref{B-6}, \eqref{B-7}, and \eqref{B-8} to see
	\[
	\|V(t)\|_{\infty,2} \ \ge\ |v_{i_\ell}(t)|_2 \ \ge|v_{i_\ell}^{\mathrm{in}}|_2 -\varepsilon\  >\|V^{\mathrm{in}}\|_{\infty,2}-2\e.
	\]
	Since $\varepsilon>0$ is arbitrary, we obtain the desired estimate:
	\[ \|V(t)\|_{\infty,2} \ge \|V^{\mathrm{in}}\|_{\infty,2}\quad \forall~t\geq 0. \]
\end{proof}
\begin{remark}
	We can easily construct an example satisfying Proposition \ref{P2.3} as follows: let
	\[k_i = \left\lfloor\frac{i+1}{2}\right\rfloor,\quad  m_i = \frac{1}{2^{k_i+1}},\quad x_i^{\mathrm{in}} = k_i \log 2-2\log k_i,\quad v_i^{\mathrm{in}} = (-1)^{i}. \]
	Then we have
	\begin{equation*}
		\sum_{i=1}^\infty m_i = 1, \qquad \sum_{i=1}^{\infty} m_i v_i^{\mathrm{in}}=0,
	\end{equation*}
	and $(X^{\mathrm{in}}, V^{\mathrm{in}})$ satisfies \eqref{B-2-1} with $\cP_\infty=1$, and
	\[\sum_{i=1}^{\infty}m_ie^{|x_i^{\mathrm{in}}|} = \sum_{k=1}^{\infty}\frac{1}{k^2}<\infty.\]
	Therefore, we have
	\[\|V(t)\|_{\infty,2} = 1,\quad \forall~t\geq 0.\]
\end{remark}
Even though we do not have the decay estimate of $\|V(t)\|_{\infty,p}$ in general, we can still derive decay estimates for $\|V(t)\|_{q,p}$ when $q<\infty$ as in \cite{w8,w9,w10}. We introduce the time-dependent index set $\Lambda_t$ together with an increasing function $R(t)$ defined by
\[ \Lambda_t := \{i:|x_i(t)|_p\leq R(t)\}\quad \text{where} \quad R(t)\simeq (1+t)^{\gamma}.\]
Here, the exponent $\gamma>1$ will be specified later. The following lemma provides a decay estimate for the sum of  $m_i$'s corresponding to particles located outside the ball $B_{R(t)}$. 

\begin{lemma} \label{L2.2}
 Let $(X(t),V(t))$ be a global solution to \eqref{B-2} with initial data $(X^{\mathrm{in}},V^{\mathrm{in}})\in\ell^{q,p}\times \ell^{\infty,p}$. Then, we have
	\[\sum_{i\in \Lambda_t^c}m_i\lesssim (1+t)^{-q(\gamma-1)}. \]
\end{lemma}
\begin{proof}
	Since $R(t)< |x_i(t)|$ for $i\in \Lambda_t^c$, we use Lemma \ref{L2.1} (2) to derive
	\begin{align} \label{B-9}
		\begin{aligned}
			R(t)^q\sum_{i\in \Lambda_t^c}m_i &\leq \sum_{i\in \Lambda_t^c}m_i|x_i|_p^q \leq \sum_{i=1}^\infty m_i|x_i|_p^q \leq \left(\|X^{\mathrm{in}}\|_{q,p}+t\|V^{\mathrm{in}}\|_{q,p}\right)^q\\
			&\leq 2^{q-1}\left(\|X^{\mathrm{in}}\|_{q,p}^q+t^q\|V^{\mathrm{in}}\|_{q,p}^q\right).
		\end{aligned}
	\end{align}
	We divide both sides of \eqref{B-9} by $R(t)^q$ to get
	\[\sum_{i\in \Lambda_t^c}m_i\leq \frac{2^{q-1}\Big(\|X^{\mathrm{in}}\|_{q,p}^q+t^q\|V^{\mathrm{in}}\|_{q,p}^q\Big)}{R(t)^q}\lesssim (1+t)^{-q(\gamma-1)}. \]
\end{proof}

Next, we derive the decay estimate of $\|V(t)\|_{q,p}$. In the following lemma, we show that the map $x\mapsto |x|_p^q$ is uniformly convex, which improves \eqref{A.2}. For $2\leq p\leq q$, we set
\[  G(x):=|x|_p^q, \quad c_q = 2^{2-q}>0. \]
\begin{lemma} \label{L2.3}
The function $G$ satisfies 
	\[(\nabla G(a)-\nabla G(b))\cdot (a-b)\geq c_q|a-b|_p^q\quad \forall~a,b\in\bbr^d.\]
\end{lemma}
\begin{proof} We define
\[ y(t)=b+(a-b)t, \quad  h(t):= G(y(t)). \]
First, one can easily check that the map $t\mapsto |y(t)|_p$ is convex. Moreover, $y\mapsto y^q$ is convex and increasing. This implies $h$ is convex. On the other hand, one has 
	\begin{align}\label{B-10}
		\begin{aligned}
			& (\nabla G(a)-\nabla G(b))\cdot (a-b) \\
			& \hspace{0.5cm} =h'(1)-h'(0) \geq 2\left(h(1)-h\left(\frac{1}{2}\right)\right)-2\left(h\left(\frac{1}{2}\right)-h\left(0\right)\right)\\
			& \hspace{0.5cm} = 2\left(h(1)+h(0)-2h\left(\frac{1}{2}\right)\right) =2\left(G(a)+G(b)-2G\left(\frac{a+b}{2}\right)\right).
		\end{aligned}
	\end{align}
	Recall Clarkson's inequality \cite{C}: for $p\geq 2$, 
	\[2(|x|_p^p+|y|_p^p)\leq|x+y|_p^p+|x-y|_p^p,  \quad\forall~x,y\in\bbr^d.\]
	Since the map $s \mapsto s^{q/p}$ is convex, we use Clarkson's inequality and Jensen's inequality to get
	\begin{equation*}
		\left(\left|\frac{a+b}{2}\right|_p^p+\left|\frac{a-b}{2}\right|_p^p\right)^{\frac{q}{p}}\leq \left(\frac{|a|_p^p+|b|_p^p}{2}\right)^{\frac{q}{p}}\leq \frac{|a|_p^q+|b|_p^q}{2}.
	\end{equation*}
	Therefore, we have
	\begin{align}\label{B-11}
		\begin{aligned}
			& G(a)+G(b)-2G\left(\frac{a+b}{2}\right) \\
			& \hspace{1cm} = 2\left(\frac{|a|_p^q+|b|_p^q}{2} - \left|\frac{a+b}{2}\right|_p^q\right) \\
			& \hspace{1cm} \geq 2\left(\left(\left|\frac{a+b}{2}\right|_p^p+\left|\frac{a-b}{2}\right|_p^p\right)^{\frac{q}{p}}-\left|\frac{a+b}{2}\right|_p^q\right) \geq 2\left|\frac{a-b}{2}\right|_p^{q},
		\end{aligned}
	\end{align}
	where we used
	\[(x+y)^\alpha-x^\alpha\geq y^\alpha,\quad \forall~x,y>0\]
	for $\alpha=q/p\geq 1$.
	We combine \eqref{B-10} and \eqref{B-11} to conclude
	\begin{align*}
		(\nabla G(a)-\nabla G(b))\cdot (a-b)\geq 2^{2-q}|a-b|_p^q.
	\end{align*}
\end{proof}
Now, we provide the decay estimate of $\|V(t)\|_{q,p}$. 	Note that we need to assume $\beta<1$ for the following result.
\begin{proposition}\label{P2.4}
Let $\beta \in (0, 1)$ and $2\leq p\leq \bar{q}$. Suppose that $(X^{\mathrm{in}},V^{\mathrm{in}})\in\ell^{\bar{q},p}\times \ell^{\infty,p}$ satisfy
\[\sum_{i=1}^\infty m_iv_i^{\mathrm{in}}=0,\]
and let $(X(t),V(t))$ be a global solution to \eqref{B-2} with initial data $(X^{\mathrm{in}}, V^{\mathrm{in}})$. Then for any $q\in[p,\bar{q}]$ and any $\gamma\in(1,\frac{1}{\beta})$,  we have\[\|V(t)\|_{q,p}\lesssim (1+t)^{-\frac{\bar{q}}{q}(\gamma-1)}.\]
\end{proposition}	
\begin{proof}
	The proof is based on a combination of the time-varying effective domain method proposed in  \cite{w9, w8,w10} and Lemma \ref{L2.3}, and is given in Appendix \ref{App-B}.
\end{proof}

\begin{remark} \label{R2.5}
The preceding analysis reveals a fundamental obstruction in establishing long-time stability for the classical infinite ICS model in a fully noncompact spatial setting. Although we can obtain decay of velocity moment, Proposition \ref{P2.3} shows that the velocity diameter may fail to decay in time even when the spatial configuration satisfies strong moment conditions and is not excessively dispersed. This indicates that the alignment mechanism based on pairwise distance interactions is insufficient to guarantee uniform-in-time stability once particles are allowed to spread throughout the whole Euclidean space.
\end{remark}
	\subsection{Counterexample}\label{sec:2.3}
	In this subsection, we provide a counterexample to show that uniform-in-time stability can be destroyed even when two initial data are close enough in a spatially extended setting.
For this, we recall the KCS model:
\begin{equation} \label{B-11-1}
	\begin{cases} 
		\displaystyle \partial_tf+ \nabla_x \cdot (v f) + \nabla_v \cdot (L[f]f)=0, \quad t > 0,~~(x, v) \in {\mathbb R}^{2d},  \vspace{6pt}\\
		\displaystyle L[f](t,x,v) :=-\kappa\int_{\mathbb{R}^{2d}} \phi(|x-x_*|_p) \left(v-v_*\right)  f(t,x_*,v_*) dx_* d v_*,\vspace{6pt} \\
		\displaystyle f \Big|_{t = 0}  =f^{\mathrm{in}}.
	\end{cases}
\end{equation}
To consider the distance between measure-valued solutions, we employ the Wasserstein distance to compute the distance between two measures $\mu, \nu\in\mathcal{P}(\bbr^{2d})$ as follows: for $p\in[1,\infty]$ and $q\in[1,\infty)$, $W_{q,p}$ is defined as
\begin{equation*}
	W_{q,p}(\mu,\nu):=\inf_{\pi\in\Pi(\mu,\nu)}\left(\int_{\bbr^{2d}\times\bbr^{2d}}|z-z^*|_p^q\pi(dz,dz^*)\right)^{1/q},\quad W_p(\mu,\nu):=W_{p,p}(\mu,\nu).
\end{equation*}
For $q=\infty$, we define $W_{\infty,p}:=\lim_{q\to\infty}W_{q,p}$ so that 
\begin{equation*}
	W_{\infty,p}(\mu,\nu):= \inf_{\pi\in\Pi(\mu,\nu)}\sup_{(z,z^*)\in \operatorname{supp}\pi}|z-z^*|_p.
\end{equation*}
For further details on Wasserstein distances, we refer to \cite{k8}. Now, we recall uniform-in-time stability estimate for \eqref{B-11-1}. First, we set 
\[\cD_X(t):=\sup_{x,y\in\mathrm{supp}_x\mu_t\cup \mathrm{supp}_x \nu_t}|x-y|_p,\quad\cD_V(t):=\sup_{v,w\in\mathrm{supp}_v\mu_t\cup \mathrm{supp}_v\nu_t}|v-w|_p.\]
\begin{proposition}\emph{\cite{HKZ}} \label{P2.5}
Suppose that the initial probability measures $\mu_0, \nu_0\in\cP(\bbr^{2d})$ have a compact support with 
\[\int_{\bbr^{2d}}v~ \mu_0(dx,dv)=\int_{\bbr^{2d}}v~ \nu_0(dx,dv) \quad \mbox{and} \quad 
\kappa>\frac{\cD_V(0)}{\int_{\cD_X(0)}^{\infty}\phi(s)ds},
\]
and let $\mu_t$ and $\nu_t$ be the measure-valued solutions to \eqref{B-11-1} with initial measure $\mu_0$ and $\nu_0$, respectively. Then there exists a nonnegative constant $G$ independent of $t$ such that 
\[\sup_{t\geq 0}W_p(\mu_t,\nu_t)\leq G~W_p(\mu_0, \nu_0).\]
\end{proposition}
 Proposition \ref{P2.5} shows that a small perturbation of the initial data can be controlled uniformly in time, provided that the two initial measures have compact and similar supports, where similarity of supports is understood in terms of comparable diameters. This result, however, cannot be extended to the case where the initial position supports  are not contained in a compact set, since then $\cD_X(0)=\infty$. Furthermore, once a measure has noncompact support, the notion of similarity of supports, understood in terms of their diameters as in the compactly supported case, is no longer meaningful. In what follows, we explain why an analogous statement to Proposition \ref{P2.5} fails in the noncompact setting. We begin with two lemmas.
\begin{lemma} \label{L2.4}
	Consider two probability measures $\mu$ and $\nu$ in $\bbr$:
	\[\mu = \sum_{i=1}^{\infty} m_i \delta_{x_i},\quad \nu=\sum_{i=1}^{\infty} m_i \delta_{\bar{x}_i}.\]
	If $\{x_i\}_{i=1}^\infty$ and $\{\bar{x}_i\}_{i=1}^\infty$ are sets of distinct atoms with the following property:
	\[x_i<x_j\quad \Longleftrightarrow \quad \bar{x}_i<\bar{x}_j,\]
	then we have
	\[W_2^2(\mu,\nu) = \sum_{i=1}^\infty m_i|x_i-\bar{x}_i|^2.\]
\end{lemma}
\begin{proof}
	Let $F(t):=\mu\big((-\infty, t]\big), G(t):=\nu\big((-\infty, t]\big)$ and consider the generalized inverse:
	\[F^{-1}(t):=\inf \{x\in\bbr: F(x)>t\}.\]
	From \cite[Section 2.2]{V2}, we have
	\begin{equation} \label{B-11-2}
		W_2^2(\mu,\nu) = \int_0^1 |F^{-1}(s) - G^{-1}(s)|^2 ~ds.
	\end{equation}
	Now, for each $i\in\bbn$, we define 
	\[S_i^- := \sum_{j:~x_j<x_i}m_j = \sum_{j:~\bar{x}_j<\bar{x}_i} m_j, \quad S_i^+ :=S_i^- + m_i,\]
	then 
	\begin{equation*}
		F^{-1}(t) = x_i,\quad G^{-1}(t) = \bar{x}_i\quad  \text{for}\quad t\in [S_i^-,S_i^+).
	\end{equation*}
	Since $\{[S_i^-,S_i^+)\}_{i=1}^\infty$ forms a partition of $[0,1)$ up to a set of measure zero, we use \eqref{B-11-2} to derive
	\begin{align*}
		W_2^2(\mu, \nu) = \sum_{i=1}^\infty \int_{S_i^-}^{S_i^+}|F^{-1}(s)-G^{-1}(s)|^2 ds= \sum_{i=1}^\infty \int_{S_i^-}^{S_i^+}|x_i-\bar{x}_i|^2 ds= \sum_{i=1}^\infty m_i|x_i-\bar{x}_i|^2.
	\end{align*}
\end{proof}
In the following lemma, we deal with the ordering principle of the ICS model. 
\begin{lemma} \label{L2.5}
	Consider the ICS model on the real line $\bbr$. Suppose that the initial data $(X^\mathrm{in}, V^{\mathrm{in}})$ satisfies 
	\[x_i^{\mathrm{in}}<x_j^{\mathrm{in}}\quad \Longrightarrow\quad v_i^{\mathrm{in}}<v_j^{\mathrm{in}},\]
	and the set $X^{\mathrm{in}} = \{x_i^{\mathrm{in}}\}_{i=1}^\infty$ has no limit point in $\bbr$. We say that a pair $(i,j)$ is adjacent if 
	\[
	x_i<x_j\quad\text{and}\quad 
	\{k\in\mathbb N: x_i<x_k<x_j\}=\emptyset.
	\]
	Let $(X(t), V(t))$ be a solution to \eqref{B-2} with initial data $(X^{\mathrm{in}}, V^{\mathrm{in}})$. Define 
	\[
	t^*:=\inf\{t>0:\exists ~\text{adjacent pair }(i,j)~\text{with }v_i(t)=v_j(t)\}.
	\]
	Suppose we have $t^*>0$, then the following two assertions hold:
	\begin{enumerate}
		\item $t^*=\infty$ holds, i.e., for all adjacent pairs $(i,j)$, $v_i(t)<v_j(t)$ for all $t\geq 0$.\\
		\item The position ordering of the particles does not change for all $t\geq 0$:
		\[x_i^{\mathrm{in}}<x_j^{\mathrm{in}}\quad \Rightarrow\quad x_i(t)<x_j(t)\quad \forall~t\geq 0.\]
	\end{enumerate}
\end{lemma}
\begin{proof}
	(1) Suppose $t^*<\infty$. Then for $0\le t<t^*$ and for any adjacent pair $(i,j)$, we have
		\[\frac{d}{dt}(x_j(t)-x_i(t)) = v_j(t)-v_i(t)>0,\]
		and 
		\begin{equation} \label{B-11-2-0}
			x_j(t)-x_i(t) = x_j^{\mathrm{in}} - x_i^{\mathrm{in}}+\int_0^t (v_j(s)-v_i(s))ds>0.
		\end{equation}
		Therefore, none of the adjacent particles collide with each other and the position ordering remains the same as the initial ordering on $[0,t^*)$. Now, fix an adjacent pair $(i,j)$ with 
		\begin{equation} \label{B-11-2-1}
			v_i(t^*) = v_j(t^*).
		\end{equation}
		By \eqref{B-2}, we have
		\begin{align*}
			\dot v_j-\dot v_i &=\kappa\sum_{k=1}^\infty m_k \phi(|x_k-x_j|)(v_k-v_j)
			-\kappa\sum_{k=1}^\infty m_k \phi(|x_k-x_i|)(v_k-v_i) \\
			&=-\kappa\left(\sum_{k=1}^\infty m_k\phi(|x_k-x_j|)\right)(v_j-v_i)
			+\kappa\sum_{k=1}^\infty m_k\Big(\phi(|x_k-x_j|)-\phi(|x_k-x_i|)\Big)(v_k-v_i).
		\end{align*}
		If $k=i,j$ and $0\le t<t^*$, one can easily check
		\[\Big(\phi(|x_k-x_j|)-\phi(|x_k-x_i|)\Big)(v_k-v_i)\geq 0.\]
		For $k\neq i,j$ and $0\le t<t^*$, $x_k(t)$ satisfies
		\[x_k(t)<x_i(t)<x_j(t)\quad\text{and}\quad v_k(t)<v_i(t)<v_j(t),\]
		or
		\[x_i(t)<x_j(t)<x_k(t)\quad\text{and}\quad v_i(t)<v_j(t)<v_k(t).\]
		In both cases, since $\phi$ is non-increasing, we have
		\[\Big(\phi(|x_k-x_j|)-\phi(|x_k-x_i|)\Big)(v_k-v_i)\geq 0, \]
		which leads to 
		\[\dot{v}_j(t)-\dot{v}_i(t)\geq -\kappa\underbrace{\left(\sum_{k=1}^\infty m_k\phi(|x_k-x_j(t)|)\right)}_{\leq 1}(v_j(t)-v_i(t))
		\geq -\kappa(v_j(t)-v_i(t)),\quad 0\le t<t^*.
		\]
		We apply Gr\"onwall's lemma to derive
		\[
		v_j(t)-v_i(t)\geq (v_j(0)-v_i(0))e^{-\kappa t}>0,
		\qquad\text{for }~0\le t<t^*.
		\]
		Then $t\to t^*$ gives
		\[
		v_j(t^*)-v_i(t^*)\geq (v_j(0)-v_i(0))e^{-\kappa t^*}>0,
		\]
		which contradicts \eqref{B-11-2-1}. Therefore, $t^*=\infty$ and for all adjacent pairs $(i,j)$, we have
		\[v_i(t)<v_j(t),\quad \forall~t\geq 0.\]
		(2) Since $t^*=\infty$, for any adjacent pair $(i,j)$, \eqref{B-11-2-0} holds for all $t\geq 0$, which gives the desired result.
	\end{proof}
	\begin{remark}
		Unlike in the finite CS model, the condition $t^*>0$ in Lemma \ref{L2.5} is not automatic in the infinite-particle setting. In the following proposition, we show that $t^*>0$ holds for the particular example constructed in Proposition \ref{P2.6} below.
	\end{remark}
In the following proposition, we construct initial probability measures with noncompact spatial support using infinitely many atoms. We show that even strong moment conditions, such as finiteness of an exponential moment, are not sufficient to guarantee uniform-in-time stability of solutions.
\begin{proposition} \label{P2.6}
	There exists a probability measure $\mu_0\in\cP(\bbr^2)$ and a sequence of probability measures $(\nu_0^n)_{n\in\bbn}\subset \cP(\bbr^2)$,
	such that:
	\begin{enumerate}
		\item $\mu_0$ and each $\nu_0^n$ have noncompact spatial support, zero total momentum, and finite exponential moment:
		\[
		\int_{\bbr^2} e^{|x|}~\mu_0(dx,dv)<\infty,\quad \int_{\bbr^2}e^{|x|}~\nu_0^n(dx,dv)<\infty,\quad \forall~n\in\bbn.
		\]
		\item $W_2(\mu_0,\nu_0^n)\xrightarrow{n\to\infty}0$.\\
		\item For every $\kappa>0$, if $\mu_t$ and $\nu_t^n$ denote the corresponding measure-valued solutions of \eqref{B-11-1} with initial data $\mu_0$ and $\nu_0^n$, then
		\[
		\frac{\sup_{t\ge0}W_2(\mu_t,\nu_t^n)}{W_2(\mu_0,\nu_0^n)}
		\xrightarrow{n\to\infty}\infty.
		\]
	\end{enumerate}
	In particular, for this fixed initial datum $\mu_0$ and for every $\kappa>0$, there does not exist a common finite constant $G$ such that for all perturbations $\nu_0$ sufficiently close to $\mu_0$,
	\[
	\sup_{t\ge0}W_2(\mu_t,\nu_t)\le G~W_2(\mu_0,\nu_0).
	\]
\end{proposition}

\begin{proof}
	Before moving further, we summarize our proof strategy in three steps.
	
\begin{itemize}
	\item Step A: First, we construct an initial measure $\mu_0$ and its perturbations $\nu_0^n$ using infinitely many atoms where
	\[W_2(\mu_0,\nu_0^n)\xrightarrow{n\to\infty}0.\]
	\item Step B: 	Next, we show that for each particle system corresponding to the measures $\mu_0$ and $\nu_0^n$, the ordering of the particles does not change over time.
	\item Step C: 	 Finally, we conclude our proof by showing 
	\[\frac{\sup_{t\ge0}W_2(\mu_t,\nu_t^n)}{W_2(\mu_0,\nu_0^n)}
	\xrightarrow{n\to\infty}\infty.\]
\end{itemize} 
Next, we proceed with our proof step by step.
	
	\vspace{.3cm}
	\noindent$\bullet$ \textbf{Step A: Construct the perturbation sequence. }For each $j\in\bbn$, set
	\[R_j:=e^{j^2},\qquad R_{-j}:=-R_j,\qquad u_j:=1-2^{-j},\qquad u_{-j} = -u_j.\]
	Let
	\[
	S:=2\sum_{j=1}^\infty e^{-2R_j},\qquad m_j:=S^{-1}e^{-2R_j}=m_{-j}\quad \text{for}\quad j\neq 0 .
	\]
	For convenience, we set
	\[ R_0=u_0=m_0=0, \]
	and define
	\[
	\mu_0
	:=\sum_{j\in\bbz} m_{j}\delta_{(R_j,u_j)}.
	\]
	Since 
	\[
	\sum_{j\in\bbz}m_j=2\sum_{j=1}^\infty m_j =1\quad \text{and}\quad \lim_{j\to\pm\infty} R_j=\pm\infty,
	\]
	$\mu_0$ is a probability measure with noncompact spatial support. By symmetry, its total momentum vanishes:
	\[
	\int_{\mathbb R^2} v~d\mu_0(x,v)
	=
	\sum_{j=1}^\infty m_{j}u_j+\sum_{j=1}^\infty m_{-j}u_{-j}=\sum_{j=1}^\infty m_{j}u_j-\sum_{j=1}^\infty m_{j}u_{j}=0,
	\]
	and has a finite exponential moment:
	\[
	\int_{\mathbb R^2} e^{|x|}~\mu_0(dx,dv)
	=
	2\sum_{j=1}^\infty m_j e^{R_j}
	=
	2S\sum_{j=1}^\infty e^{-R_j}<\infty.
	\]
	Now, for each $n\in\mathbb N$, let 
	\[\e_n:=2^{-2n}\]
	and define $\nu_0^n$ by perturbing only the $n$-th symmetric pair of $\mu_0$:
	\[
	\nu_0^n
	:=
	\mu_0  - m_n\delta_{(R_n,u_n)} - m_{-n}\delta_{(R_{-n},u_{-n})}
	+ m_n\delta_{(R_n,u_n+\e_n)} + m_{-n}\delta_{(R_{-n},u_{-n}-\e_n)}.
	\]
	Then $\nu_0^n$ is again a probability measure with noncompact spatial support. Moreover, its total momentum is still zero:
	\[
	m_n(u_n+\e_n) + m_{-n}(u_{-n}-\e_n)=0.
	\]
	Since $\nu_0^n$ has the same spatial support as $\mu_0$, it also satisfies the moment condition:
	\[
	\int_{\mathbb R^2} e^{|x|}~\nu_0^n(dx,dv)<\infty.
	\]
	Next, we estimate the initial Wasserstein distance. By coupling each atom of $\mu_0$ with the atom of $\nu_0^n$ at the same spatial position, we obtain
	\[
	W_2^2(\mu_0,\nu_0^n) \leq  m_n\e_n^2+m_{-n}\e_n^2 = 2m_n\e_n^2.
	\]
	Therefore
	\[
	W_2(\mu_0,\nu_0^n)\leq \sqrt{2m_n}~\e_n \xrightarrow{n\to\infty}0.
	\]
	
	\vspace{.5cm}
	\noindent$\bullet$ \textbf{Step B: Order relationship between particles' position and velocity.} Now, let $(x_j(t),v_j(t))_{j\in\bbz}$ and $(\bar x_j(t),\bar v_j(t))_{j\in\bbz}$ denote the particle trajectories corresponding to $\mu_0$ and $\nu_0^n$, respectively. For instance, $(x_j(t), v_j(t))$ corresponds to the position and velocity of the particle with initial data $(R_j, u_j)$ and weight $m_j$, for all $j\in\bbz$. We first show that for each of these systems, the time $t^*$ in Lemma \ref{L2.5} is strictly positive. Once this is established, Lemma \ref{L2.5} implies that the position ordering is preserved for all time in both systems. \\\\
	 Recall that Proposition \ref{P2.1} gives
	\begin{equation} \label{B-11-3}
		|v_j(t)|\leq 1,\qquad |\bar v_j(t)|\leq 1\qquad \forall~ j\in\bbz, \quad \forall~ t\geq0.
	\end{equation}
	We consider the system $(x_j(t),v_j(t))$. Since $\{R_j\}_{j\in\bbz} = \{\pm e^{j^2}\}_{j\in\bbn}\cup \{0\}$, we have 
	\[e = \inf_{j\neq k}|R_j-R_k|.\]
	Then for $0\leq t\leq \frac{e}{4}$, we use \eqref{B-11-3} to get
	\begin{equation} \label{B-11-3-1}
		|x_j(t)-x_k(t)|\geq |R_j-R_k|- 2t \geq \frac{1}{2}|R_j-R_k|, \quad \forall~j\neq k.
	\end{equation}
	Fix $j\in\bbn$, a particle index beginning on the positive side. We use \eqref{B-11-3} and \eqref{B-11-3-1} to derive
	\begin{align*}
		|\dot{v}_j(t)|&\leq \kappa\sum_{k\neq j} m_k \phi(|x_k(t)-x_j(t)|)|v_k(t)-v_j(t)|\\
		&\leq 2\kappa \sum_{k\neq j} m_k \phi(|x_k(t)-x_j(t)|)\\
		&\leq 2\kappa \sum_{k\neq j} m_k \phi\left(\frac{1}{2}|R_j-R_k|\right).
	\end{align*}
	Note that since $R_j = e^{j^2}$, we have
	\[|R_j-R_k|\geq \left(1-\frac{1}{e}\right)R_j,\quad \forall~k\neq j.\]
This yields
	\begin{equation*}
		|\dot{v}_j(t)| \leq 2\kappa \phi \left(\frac{e-1}{2e}R_j\right):=2\kappa B_j \simeq R_j^{-\beta} = e^{-\beta j^2}.
	\end{equation*}
	Now for adjacent particles on the positive side, we have
	\[v_{j+1}(0) - v_j(0) = u_{j+1} - u_j = 2^{-j-1}.\]
	Thus, for $0\leq t\leq \frac{e}{4}$, we have
	\begin{align} \label{B-11-3-0}
		\begin{aligned}
		v_{j+1}(t)-v_j(t)&\geq u_{j+1}-u_j -|v_{j+1}(t) - u_{j+1}| - |v_j(t)-u_j|\\
		&\geq 2^{-j-1}-2\kappa (B_{j+1}+B_j)t.
		\end{aligned}
	\end{align}
	Since 
	\[\frac{2^{-j-1}}{B_{j+1}+B_j} \geq \frac{2^{-j-1}}{2B_{j}}\gtrsim \frac{e^{\beta j^2}}{2^{j}}\xrightarrow{j\to\infty}\infty,\]
	we can define
	\[t_0:= \min \left\{\frac{e}{4},\frac{1}{4\kappa}\inf_{j\in\bbn}\frac{2^{-j-1}}{B_{j+1}+B_{j}}\right\}>0.\]
	Then by \eqref{B-11-3-0}, for all $j\in\bbn$ and $t\in[0,t_0)$, we have
	\begin{equation*}
		v_{j+1}(t)-v_j(t)> 0,
	\end{equation*}
	i.e., for all adjacent particles on the positive side, their velocity does not become the same for some positive time. By symmetry, the same conclusion holds on the negative side. Therefore, $t^*>0$ for the system $(x_j(t), v_j(t))$. The same argument applies to $(\bar{x}_j(t), \bar{v}_j(t))$. Hence, by Lemma \ref{L2.5}, position ordering is preserved for all time in both systems.
	
	\vspace{.3cm}
	\noindent$\bullet$ \textbf{Step C: Blow up of the ratio between whole-time distance and initial distance.} We now fix $\kappa>0$. Since
	\[
	R_n=e^{n^2}\gg 2^{2n}=\e_n^{-1}=:T_n,
	\]
	for large $n>0$, we have
	\[
	T_n\leq \frac{e-2}{4e}R_n.
	\]
	Then by \eqref{B-11-3}, it follows that for every $t\in[0,T_n]$, we have
	\begin{align*}
		|x_n(t)-x_k(t)|&\geq |R_n-R_{k}| -\int_0^t |v_n(s)|ds - \int_0^t |v_k(s)|ds\\
		& \geq \left(1-\frac{1}{e}\right)R_n - \frac{e-2}{2e}R_n = \frac{R_n}{2},\quad \forall~k\neq n,
	\end{align*}
	i.e., the particles initially located at $R_n$ (and $R_{-n}$) stay at distance at least $R_n/2$ from all the other particles, in both systems $(x_j(t),v_j(t))_{j\in\bbz}$ and $(\bar x_j(t),\bar v_j(t))_{j\in\bbz}$. Therefore, we have
	\[
	|\dot v_n(t)| \leq \kappa \sum_{j\in\bbz} m_j \phi(|x_j-x_n|)|v_j-v_n|
	\leq
	\kappa \phi\left(\frac{R_n}{2}\right)\sup_{j\in\bbn}\big(|v_j|+|v_n| \big)
	\leq2\kappa \phi\left(\frac{R_n}{2}\right),
	\]
	for all $t\in[0,T_n]$, where we again used \eqref{B-11-3} in the last inequality. The same estimates hold for $v_{-n}(t)$, $\bar{v}_n(t)$, and $\bar{v}_{-n}(t)$. Since
	\[
	\phi\left(\frac{R_n}{2}\right)
	=
	\frac{1}{\big(1+R_n^2/4\big)^{\beta/2}}
	\lesssim R_n^{-\beta}
	=
	e^{-\beta n^2}\quad \text{and}\quad	\e_n^2 = 2^{-4n},
	\]
	we have
	\[
	\lim_{n\to\infty}\frac{\phi(R_n/2)}{\e_n^2}=0.
	\]
	Therefore, for all sufficiently large $n$, we have
	\[
	|\dot{v}_n(t)|\leq 2\kappa \phi\left(\frac{R_n}{2}\right)\leq \frac{\e_n^2}{4},\quad \forall ~0\leq t\leq T_n.
	\]
	Hence for $0\leq t\leq T_n=\e_n^{-1}$, we have
	\[
	|v_n(t)-u_n|\leq \int_0^t\left|\dot{v}_n(s)\right|ds \leq \frac{\e_n^2}{4}t \leq \frac{\e_n}{4},
	\]
	and similarly
	\[
	|\bar v_n(t)-(u_n+\e_n)|\leq \frac{\e_n}{4}.
	\]
	It follows that
	\begin{align} \label{B-11-4}
		\begin{aligned}
			\bar v_n(t)-v_n(t)&\geq \left(u_n+\frac{3}{4}\e_n\right) - \left(u_n + \frac{1}{4}\e_n\right) =\frac{\e_n}{2},\quad \forall~ t\in[0,T_n].
		\end{aligned}
	\end{align}
	Likewise, for the $n$-th negative particles, we have
	\[
	v_{-n}(t)-\bar v_{-n}(t)\ge \frac{\e_n}{2}
	\qquad \forall~ t\in[0,T_n].
	\]
	Integrating \eqref{B-11-4} leads to
	\begin{equation} \label{B-11-5}
		\bar x_n(T_n)-x_n(T_n)
		=
		\int_0^{T_n} \big(\bar v_n(s)-v_n(s)\big)~ds
		\ge \frac{\e_n}{2}~T_n
		= \frac{1}{2}.
	\end{equation}
	Similarly, we have
	\begin{equation} \label{B-11-6}
		x_{-n}(T_n)-\bar x_{-n}(T_n)\ge \frac{1}{2}.
	\end{equation}
	Now let $\mathrm{Pr}:(x,v)\mapsto x$ denote the projection onto the space variable. Then we have
	\[
	W_2(\mathrm{Pr}_{\#}\mu_{T_n},\mathrm{Pr}_{\#} \nu^n_{T_n}) \leq W_2(\mu_{T_n},\nu_{T_n}^n).
	\]
	On the other hand, by Lemma \ref{L2.4} and Lemma \ref{L2.5}, we have
	\begin{align} \label{B-11-7}
		\begin{aligned}
		& W_2^2(\mathrm{Pr}_\# \mu_t, \mathrm{Pr}_{\#} \nu^n_t) \\
		& \hspace{0.5cm} =W_2^2 \left(\sum_{j\in\bbz} m_j\delta_{x_j(t)}, \sum_{j\in\bbz} m_j \delta_{\bar{x}_j(t)}\right) = \sum_{j\in\bbz} m_j|x_j(t)-\bar{x}_j(t)|^2, \quad \forall~t\geq 0.
		\end{aligned}
	\end{align}
	We combine \eqref{B-11-5}, \eqref{B-11-6}, and \eqref{B-11-7} to get
	\begin{equation*}
		W_2^2(\mathrm{Pr}_{\#}\mu_{T_n},\mathrm{Pr}_{\#} \nu^n_{T_n})\geq m_n\left(\frac{1}{2}\right)^2 + m_{-n}\left(\frac{1}{2}\right)^2 = \frac{m_n}{2}.
	\end{equation*}
	Thus
	\[
	\sup_{t\ge0}W_2(\mu_t,\nu_t^n)\geq
	W_2(\mathrm{Pr}_{\#}\mu_{T_n},\mathrm{Pr}_{\#} \nu^n_{T_n})
	\geq\sqrt{\frac{m_n}{2}}.
	\]
	Therefore, for every fixed $\kappa>0$ and all sufficiently large $n$,
	\[
	\frac{\sup_{t\ge0}W_2(\mu_t,\nu_t^n)}{W_2(\mu_0,\nu_0^n)}
	\geq \sqrt{\frac{m_n}{2}}\frac{1}{\sqrt{2m_n}\e_n}
	=\frac{1}{2\e_n}=2^{2n-1},
	\]
	which gives
	\[
	\frac{\sup_{t\ge0}W_2(\mu_t,\nu_t^n)}{W_2(\mu_0,\nu_0^n)}
	\xrightarrow{n\to\infty}\infty.
	\]
	\end{proof}
	Motivated by these limitations of original CS model in a spatially extended setting, we introduce a modified ICS-type model in a metric space $(\bbr^d, |\cdot|_p)$ in which the communication weight depends on the sum of pairwise spatial moment:
	\begin{align}
		\begin{cases} \label{B-12}
			\displaystyle \frac{dx_i}{dt} = v_i, \quad t>0,\quad i \in {\mathbb N},\vspace{6pt}\\
			\displaystyle\frac{dv_i}{dt} = \kappa\sum_{j=1}^\infty m_{j}\phi(\bw(X))(v_j-v_i),\vspace{6pt}\\
			\displaystyle\bw(X)=\sum_{i,j=1}^\infty m_im_j|x_j-x_i|_p,	\vspace{6pt}\\
			\displaystyle	(x_i,v_i) \Big|_{t = 0} =(x_i^{\mathrm{in}},v_i^{\mathrm{in}}),\quad \sum_{j=1}^\infty m_{j}=1.
		\end{cases}
	\end{align}
	This modification preserves the alignment structure while enabling the uniform stability results in fully noncompact regimes. In the remainder of this paper, we therefore focus on this ICS-type model and develop a uniform-in-time stability theory together with the corresponding mean-field limit:
	\begin{equation} 
		\begin{cases} \label{B-13}
			\displaystyle \partial_tf+ \nabla_x \cdot (v f) + \nabla_v \cdot (L[f]f)=0, \quad t > 0,~~(x, v) \in {\mathbb R}^{2d},  \vspace{6pt}\\
			\displaystyle L[f](t,x,v)=-\kappa\int_{\mathbb{R}^{2d}} \phi( \bw[f](t)) \left(v-v_*\right)  f_*d x_* d v_*,\vspace{6pt} \\
			\displaystyle \bw[f](t)=\int_{\mathbb{R}^{4d}}|x-x_*|_pff_*d x d vd x_* d v_*,\vspace{6pt} \\
			\displaystyle f \Big|_{t = 0}  =f^{\mathrm{in}}.
		\end{cases}
	\end{equation}
	Before we proceed to the next subsection, we recall a system of dissipative differential inequalities that will play a key role in the subsequent analysis.
\begin{lemma}{\emph{\cite{HKZ}}}\label{L2.6}
Let $\mathcal{X}$ and $\mathcal{V}$ be Lipschitz continuous functions satisfying the system of dissipative differential inequalities (SDDI):
	\begin{align*}
		\left|\frac{d\mathcal{X}}{d t}\right|\le \mathcal{V}, \quad \frac{d\mathcal{V}}{d t}\le -\alpha\mathcal{V}+\gamma e^{-\alpha t} \mathcal{X}, \quad \mbox{a.e.}~t > 0,
	\end{align*}
	where $\alpha$ and $\gamma$ are positive constants. We set 
	\[
	L_0:=\max\left\{1,\frac{2\gamma}{\alpha e}\right\}+\frac{8\gamma}{\alpha^3e^3} \quad \mbox{and} \quad L_1:= \frac{2L_0}{\alpha}+L_0. 
	\]
	Then, the following assertions hold:
	\begin{enumerate}
		\item	
		The functionals $\mathcal{X}$ and $\mathcal{V}$ satisfy
		\begin{align*}
			\mathcal{X}(t)\le \frac{2L_0}{\alpha}\left(\mathcal{X}(0)+\mathcal{V}(0)\right), \quad \mathcal{V}(t)\le L _0\left(\mathcal{X}(0)+\mathcal{V}(0) \right)e^{-\alpha t}, \quad \mbox{a.e.}~~t > 0.
		\end{align*}
		\item	
		Uniform-in-time boundedness estimate holds:
		\begin{align*}
			\sup_{0 \leq t < \infty} (\mathcal{X}(t)+\mathcal{V}(t) ) \le L_1 \left(\mathcal{X}(0)+\mathcal{V}(0)\right).
		\end{align*}
	\end{enumerate}	
\end{lemma}

	\subsection{Description of main results}\label{sec:2.4}
	In this subsection, we provide our main results on the uniform stability of the ICS-type model \eqref{B-12} and the KCS-type model \eqref{B-13}. As a byproduct, we also provide the uniform mean-field limit from ICS-type model to KCS-type model with a fully noncompact support. As an application of Lemma \ref{L2.6}, we have the first main result on uniform stability estimate.
	\begin{theorem} \label{T2.1}
	For $p\in [2,\infty)$, we assume that two sets of initial data and system parameters satisfy
		\begin{align*}
		\begin{aligned}
			& (X^{\mathrm{in}},V^{\mathrm{in}}),~~(\bar{X}^{\mathrm{in}},\bar{V}^{\mathrm{in}})\in \ell^{p,p}\times \ell^{p,p}, \quad \sum_{i=1}^\infty m_i v_i^{\mathrm{in}}=\sum_{i=1}^\infty m_i \bar{v}_i^{\mathrm{in}}=0,\\
			& \kappa>\frac{4}{\tilde{c}_p}\max \left\{\frac{\|V^{\mathrm{in}}\|_{p}}{\int_{2\|X^{\mathrm{in}}\|_{p}}^\infty \phi(s)ds},\frac{\|\bar{V}^{\mathrm{in}}\|_{p}}{\int_{2\|\bar{X}^{\mathrm{in}}\|_{p}}^\infty\phi(s)ds}\right\}, \quad \tilde{c}_p=\frac{2^{2-p}}{p},
		\end{aligned}
		\end{align*} 
and let $(X(t), V(t))$, $(\bar{X}(t), \bar{V}(t))$ be global solutions to \eqref{B-12}. Then, there exists a constant $C>0$ independent of time $t$ such that
		\begin{align*}
		\sup_{0 \leq t < \infty} \Big( \|X(t)-\bar{X}(t)\|_{p}+\|V(t)-\bar{V}(t)\|_{p} \Big) \le C \Big (\|X^{\mathrm{in}}-\bar{X}^{\mathrm{in}}\|_{p}+\|V^{\mathrm{in}}-\bar{V}^{\mathrm{in}}\|_{p} \Big ).
		\end{align*}
	\end{theorem}
	\begin{proof}
	Although the detailed proof can be found in Section \ref{sec:3}, we briefly summarize a proof strategy in three steps:
	\vspace{0.1cm}
		\begin{itemize}
			\item
			Step A:~We use SDDI technique to show that there exist two $x_{q,p}^{\infty}$ and $\bar{x}_{q,p}^{\infty}$ (we assume that $x_{q,p}^{\infty}\geq \bar{x}_{q,p}^{\infty}$) such that the following two assertions hold.
			\begin{enumerate}
				\item (Spatial cohesion): There exists an $x_{q,p}^{\infty}<\infty$ such that 
			\[ \sup_{0 \leq t < \infty} \|X(t) \|_{q,p}<x_{q,p}^{\infty},\quad \sup_{0 \leq t < \infty} \|\bar{X}(t) \|_{q,p}<\bar{x}_{q,p}^{\infty}.\]
			Here, $x_{q,p}^{\infty}$ and $\bar{x}_{q,p}^{\infty}$ are nonnegative constants defined implicitly by the following relations:
			\[\|V^{\mathrm{in}}\|_{q,p}=\frac{\kappa \tilde{c}_q}{4} \int_{2\|X^{\mathrm{in}}\|_{q,p}}^{2x_{q,p}^{\infty}}\phi(s)d s,\quad \|\bar{V}^{\mathrm{in}}\|_{q,p}=\frac{\kappa \tilde{c}_q}{4} \int_{2\|\bar{X}^{\mathrm{in}}\|_{q,p}}^{2\bar{x}_{q,p}^{\infty}}\phi(s)d s.\]
			\item  (Velocity alignment): The velocity moment converges to zero exponentially fast, i.e., 
			\[\|V\|_{q,p}\leq e^{-\frac{\kappa \tilde{c}_q}{2} \phi(2x_{q,p}^{\infty})t}\|V^{\mathrm{in}}\|_{q,p},\quad \|\bar{V}\|_{q,p}\leq e^{-\frac{\kappa \tilde{c}_q}{2} \phi(2\bar{x}_{q,p}^{\infty})t}\|\bar{V}^{\mathrm{in}}\|_{q,p}.\]
			\end{enumerate}
			For further details, we refer to Proposition \ref{P3.1}.
			\vspace{0.1cm}
			\item
			Step B:~We combine weak flocking estimate in Proposition \ref{P3.1}  to see
			\[
			\hspace{1.5cm}
			\begin{cases}
			\displaystyle	\frac{d}{dt}	\|X-\bar{X}\|_{p} \le \|V-\bar{V}\|_{p}, \quad \mbox{a.e.,}~~ t > 0, \\[1em]
                         \displaystyle \frac{d}{dt}\|V-\bar{V}\|_p \le - \kappa 2^{1-p}\phi(2x_p^\infty)\|V-\bar{V}\|_p+4\kappa \|\phi\|_{\mathrm{Lip}}e^{-\frac{\kappa \tilde{c}_p}{2} \phi(2\bar{x}_{p}^{\infty})t}\|\bar{V}^{\mathrm{in}}\|_{p}\|X-\bar{X}\|_{p}. 
                         \end{cases}
                         \]
			For further details, we refer to Section \ref{sec:3.2}. 
			\vspace{0.1cm}
			\item
			Step C:~We combine the above two steps together with Lemma \ref{L2.6} to get  the uniform-in-time stability of ICS-type model.
			\end{itemize}
		 For further details, we refer to Section \ref{sec:3}. 
	\end{proof}
	\begin{remark}\label{NewR2.7}
	Suppose that the initial data satisfy $(X^{\mathrm{in}}, V^{\mathrm{in}}) \in \ell^{p,p} \times \ell^{p,p}$. Then, for the ICS-type model \eqref{B-12}, there exists a unique global solution $(X(t), V(t)) \in C^1([0, \infty); \ell^{p,p} \times \ell^{p,p})$. This follows from the standard Cauchy--Lipschitz theory, since the vector field of \eqref{B-12} has linear growth and, as established in Step B of Theorem \ref{T2.1}, is locally Lipschitz continuous.
	\end{remark}
Next, we provide the following uniform-in-time mean-field limit results.
\begin{theorem} \label{T2.2}
For $p\in[2,\infty)$, we assume that the initial datum $\mu_0\in \cP(\bbr^{2d})$ satisfy
	\begin{equation*}
		\int_{\bbr^{2d}} \Big( |x|_p^p+|v|_p^p \Big)~ \mu_0(dx,dv)<\infty,\quad \int_{\bbr^{2d}}v~ \mu_0(dx,dv)=0,\quad \|V_{\mu_0}\|_{p}< \frac{\kappa \tilde{c}_p}{4} \int_{2\|X_{\mu_0}\|_{p}}^{\infty}\phi(s)d s.
	\end{equation*}
	Then there exists a unique measure-valued solution $\mu_t\in L^\infty([0,\infty);\cP(\bbr^{2d}))$ to \eqref{B-13} with initial data $\mu_0$. Also, $\mu_t$ can be approximated by $\mu_t^n$, the measure derived by the infinite particle system, in Wasserstein distance uniformly in time:
	\[\lim_{n\to\infty} \sup_{0 \leq t < \infty}W_p(\mu_t^n,\mu_t)=0.\]
\end{theorem}
\begin{proof}
We use the uniform-in-time stability results in Theorem \ref{T2.1} and particle-in-cell method to find the desired estimates. For further details, we refer to Section \ref{sec:4.1}. 
	\end{proof}
Finally, we present our last main result on the uniform stability result of the KCS-type model \eqref{A-1}.
\begin{theorem} \label{T2.3}
For $p\in[2,\infty)$, we assume that initial data $\mu_0, \nu_0 \in \cP(\bbr^{2d})$ satisfy a set of conditions:
	\begin{align*}
	\begin{aligned}
		&\int_{\bbr^{2d}} \Big( |x|_p^p+|v|_p^p \Big)~ \mu_0(dx,dv)<\infty,\quad \int_{\bbr^{2d}} \Big( |x|_p^p+|v|_p^p \Big)~ \nu_0(dx,dv)<\infty,\\
		&	\int_{\bbr^{2d}}v~ \mu_0(dx,dv)=\int_{\bbr^{2d}}v~ \nu_0(dx,dv)=0,\quad \kappa>\frac{4}{\tilde{c}_p}\max \left\{\frac{\|V_{\mu_0}\|_{p}}{\int_{2\|X_{\mu_0}\|_{p}}^\infty \phi(s)ds},\frac{\|V_{\nu_0}\|_{p}}{\int_{2\|X_{\nu_0}\|_{p}}^\infty\phi(s)ds}\right\},
	\end{aligned}
\end{align*}
	and let $\mu_t, \nu_t\in L^\infty([0,\infty);\cP(\bbr^{2d}))$ be global solutions to \eqref{B-13} with initial data $\mu_0, \nu_0$, respectively. Then, there exists a constant $C>0$ such that  
	\[ \sup_{0 \leq t < \infty} W_p(\mu_t, \nu_t)\leq C W_{p}(\mu_0, \nu_0).\]
\end{theorem}
\begin{proof}
	We combine Theorem \ref{T2.1} and Theorem \ref{T2.2} to find the desired estimates. For further details, we refer to Section \ref{sec:4.2}. 
\end{proof}
\begin{remark}\label{NewR2.8}
	In Theorems \ref{T2.2} and \ref{T2.3}, we have established the global existence and uniqueness of solutions to the KCS-type model \eqref{B-13}. Specifically, the global existence is obtained via a particle-in-cell approximation scheme, while the uniqueness follows from the stability estimates established therein.
\end{remark}
\section{Uniform-in-time stability} \label{sec:3}
\setcounter{equation}{0} 
In this section, we study the uniform stability with respect to initial data for the ICS-type model on the sender network. Recall the Cauchy problem for the ICS-type model in  a metric space $(\bbr^d, |\cdot|_p)$:
\begin{align}
	\begin{cases} \label{C-1}
		\displaystyle \frac{dx_i}{dt} = v_i, \quad t>0,\quad i \in {\mathbb N},\vspace{6pt}\\
		\displaystyle\frac{dv_i}{dt} = \kappa\sum_{j=1}^\infty m_{j}\phi(\bw(X))(v_j-v_i),\vspace{6pt}\\
		\displaystyle\bw(X)=\sum_{i,j=1}^\infty m_im_j|x_j-x_i|_p,	\vspace{6pt}\\
		\displaystyle	(x_i,v_i) \Big|_{t = 0} =(x_i^{\mathrm{in}},v_i^{\mathrm{in}}),\quad \sum_{j=1}^\infty m_{j}=1.
	\end{cases}
\end{align}
Note that the H\"{o}lder inequality and $\sum_{j=1}^\infty m_{j}=1$ imply
\begin{align}\label{C-2}
	\bw(X)\le \Bigg(\sum_{i,j=1}^\infty m_im_j|x_j-x_i|_p^q\Bigg)^{\frac{1}{q}}\le 2\Bigg(\sum_{i=1}^\infty m_i|x_i|_p^q\Bigg)^{\frac{1}{q}}=2\|X\|_{q,p}
	\end{align}
for any $q>1$. Also, one can easily check that the total weighted momentum is conserved along \eqref{C-1}:
\begin{equation} \label{C-2-0}
	\frac{d}{dt}\Bigg(\sum_{i=1}^\infty m_i v_i \Bigg)= \kappa\sum_{i,j=1}^{\infty}m_i m_j\phi(\bw (X))(v_j-v_i)=0.
\end{equation}
Furthermore, as in the original ICS model \eqref{B-2} discussed in Remark \ref{R2.2}, we may assume without loss of generality that
\begin{equation*}
	\sum_{i=1}^\infty m_iv_i^{\mathrm{in}}=0.
\end{equation*}
\subsection{Preparatory estimates}\label{sec:3.1}
In this subsection, we provide several a priori estimates for the ICS-type model. First, we derive the following SDDI.
\begin{lemma}\label{L3.1}
		Suppose $2\leq p \leq q $. Let $(X(t), V(t))$ be a global solution to \eqref{C-1} with initial data $(X^{\mathrm{in}},V^{\mathrm{in}})\in \ell^{q,p}\times \ell^{q,p}$ satisfying the zero momentum condition:
\[ \displaystyle \sum_{i=1}^\infty m_i v_i^{\mathrm{in}}= 0.  \]
Then, we have the following estimates:
		\begin{align}\label{SDDI}
			\begin{cases}
				\displaystyle \Big| \frac{d}{dt}\|X\|_{q,p} \Big| \leq \|V\|_{q,p}, \quad \mbox{a.e.,}~t > 0, \vspace{6pt}\\
								\displaystyle\frac{d}{dt}\|V\|_{q,p}\leq -\frac{\kappa \tilde{c}_q}{2} \phi(2\|X\|_{q,p})\|V\|_{q,p}.
			\end{cases}
		\end{align}
		Here, $\tilde{c}_q=\frac{2^{2-q}}{q}$.
\end{lemma}
\begin{proof}
\noindent (i)~The first inequality is the same as \eqref{A.3} in the proof of  Lemma \ref{L2.1}, so we omit it here. \newline

\noindent (ii)~It follows from \eqref{C-2-0} that 
\begin{equation} \label{C-2-1}
\sum_{i=1}^\infty m_iv_i(t) = 0, \quad t \geq 0.
\end{equation}
Next, we claim that 
	\begin{equation} \label{C-3}
	\sum_{i,j=1}^\infty m_i m_j |v_i-v_j|_p^q
	\ge
	\sum_{i=1}^\infty m_i |v_i|_p^q.
	\end{equation}
{\it Proof of \eqref{C-3}}: Note that for each $i$, we use \eqref{C-2-1} to find 
	\[
	v_i = v_i-\sum_{j=1}^\infty m_j v_j
	= \sum_{j=1}^\infty m_j (v_i-v_j).
	\]
Since the map $z\mapsto |z|_p^q$ is convex for $q\ge1$,  we use the above relation and Jensen's inequality to get
	\[
	|v_i|_p^q
	=
	\left|\sum_{j=1}^\infty m_j (v_i-v_j)\right|_p^q
	\le
	\sum_{j=1}^\infty m_j |v_i-v_j|_p^q.
	\]
	We multiply the above relation by $m_i$ and sum up the resulting relation over $i$ to find the desired inequality \eqref{C-3}.  On the other hand, recall that 
\[  G(x):=|x|_p^q, \quad c_q = 2^{2-q}>0. \]
By Lemma \ref{L2.3}, one has 
\begin{equation} \label{C-3-1}
(\nabla G(a)-\nabla G(b))\cdot (a-b)\geq c_q |a-b|_p^q\quad \forall~a,b\in\bbr^d.
\end{equation}
We use  \eqref{C-2}  and \eqref{C-3-1} to get
\begin{align} \label{C-4}
		\begin{aligned}
			\frac{d}{dt}\sum_{i=1}^\infty m_i|v_i|_p^q &= -\frac{\kappa}{2} \sum_{i,j=1}^\infty m_im_j\phi(\bw(X))(\nabla G(v_i)-\nabla G(v_j))\cdot(v_i-v_j)\\
			&\leq -\frac{\kappa c_{q}}{2}\sum_{i,j=1}^\infty m_im_j\phi(\bw(X))|v_i-v_j|_p^q\\
			&\leq -\frac{\kappa c_{q}}{2}\phi(2\|X\|_{q,p})\sum_{i,j=1}^\infty m_im_j|v_i-v_j|_p^q\\
			&\leq-\frac{\kappa c_{q}}{2}\phi(2\|X\|_{q,p})\sum_{i=1}^\infty m_i|v_i|_p^q.
		\end{aligned}
	\end{align}
We divide both sides of \eqref{C-4} by $q\|V\|_{q,p}^{q-1}$ to obtain $\eqref{SDDI}_2$.
\end{proof}

Now, we are ready to provide the weak flocking dynamics of the ICS-type model \eqref{C-1}.
\begin{proposition}\label{P3.1}
For $2\leq p\leq q$, we assume that initial data and system parameters satisfy
\begin{align*}
(X^{\mathrm{in}},V^{\mathrm{in}})\in \ell^{q,p}\times \ell^{q,p}, \quad \sum_{i=1}^\infty m_i v_i^{\mathrm{in}}=0,\quad \|V^{\mathrm{in}}\|_{q,p}< \frac{\kappa \tilde{c}_q}{4} \int_{2\|X^{\mathrm{in}}\|_{q,p}}^{\infty}\phi(s)d s,
\end{align*} 
and let $(X(t), V(t))$ be a global solution to \eqref{C-1}. Then, the following assertions hold. 
\vspace{0.1cm}
\begin{enumerate}
	\item (Spatial cohesion): There exists a positive constant $x_{q,p}^{\infty}<\infty$ such that 
	\[\|X\|_{q,p}<x_{q,p}^{\infty}.\]
	Here, $x_{q,p}^{\infty}$ is the largest positive number such that 
	\[\|V^{\mathrm{in}}\|_{q,p}=\frac{\kappa \tilde{c}_q}{4} \int_{2\|X^{\mathrm{in}}\|_{q,p}}^{2x_{q,p}^{\infty}}\phi(s)d s.\]
	\item  (Velocity alignment): The velocity moment converges to zero exponentially fast:
	\[\|V\|_{q,p}\leq e^{-\frac{\kappa \tilde{c}_q}{2} \phi(2x_{q,p}^{\infty})t}\|V^{\mathrm{in}}\|_{q,p}, \quad t \geq 0.\]
\end{enumerate}
\end{proposition}
\begin{proof}
Since the proof is based on the standard approach with the system of differential dissipative inequalities (SDDI) in \cite{k2}, we leave it in Appendix \ref{App-C}.
\end{proof}
\subsection{Proof of Theorem \ref{T2.1}}\label{sec:3.2}
In this subsection, we derive the uniform-in-time stability of ICS-type model in two steps. \newline

\noindent $\bullet$~Step A (Derivation of SDDI):~Let $p\in [2,\infty)$, and we assume that two sets of initial data and system parameters satisfy
\begin{align*}
\begin{aligned}
& (X^{\mathrm{in}},V^{\mathrm{in}}),~~(\bar{X}^{\mathrm{in}},\bar{V}^{\mathrm{in}})\in \ell^{p,p}\times \ell^{p,p}, \quad \sum_{i=1}^\infty m_i v_i^{\mathrm{in}}=\sum_{i=1}^\infty m_i \bar{v}_i^{\mathrm{in}}=0, \\
& \kappa>\frac{4}{\tilde{c}_p}\max \left\{\frac{\|V^{\mathrm{in}}\|_{p}}{\int_{2\|X^{\mathrm{in}}\|_{p}}^\infty \phi(s)ds},\frac{\|\bar{V}^{\mathrm{in}}\|_{p}}{\int_{2\|\bar{X}^{\mathrm{in}}\|_{p}}^\infty\phi(s)ds}\right\},
\end{aligned}
\end{align*} 
and let $(X(t), V(t))$, $(\bar{X}(t), \bar{V}(t))$ be global solutions to \eqref{C-1}. Then, we claim that the following estimates hold.
	\begin{equation}\label{C-5}
\begin{cases}
\displaystyle \Big| \frac{d}{dt}\|X-\bar{X}\|_{p} \Big| \leq \|V-\bar{V}\|_{p}, \quad \mbox{a.e.,}~~t > 0, \vspace{8pt}\\
\displaystyle \frac{d}{dt}\|V-\bar{V}\|_p\leq - \kappa 2^{1-p}\phi(2x_p^\infty)\|V-\bar{V}\|_p \vspace{6pt}\\
\displaystyle \hspace{2.5cm} +~ 4\kappa \|\phi\|_{\mathrm{Lip}}e^{-\frac{\kappa \tilde{c}_p}{2} \phi(2\bar{x}_{p}^{\infty})t}\|\bar{V}^{\mathrm{in}}\|_{p}\|X-\bar{X}\|_{p}.
\end{cases}
	\end{equation}
We further split the proof of the above estimates into two cases.\newline

\noindent $\star$ Case A (Derivation of the first inequality in \eqref{C-5}): We first claim that for $q\in [p,\infty)$, we have
\begin{equation} \label{C-6}
\Big| \frac{d}{dt}\|X-\bar{X}\|_{q,p} \Big| \leq \|V-\bar{V}\|_{q,p}.
\end{equation}
{\it Proof of \eqref{C-6}}: Since
\begin{align*}
	\frac{d}{dt}(x_i^k-\bar{x}_i^k) = v_i^k-\bar{v}_i^k,
\end{align*}
we have
\begin{align*}
\Big| \frac{d}{dt}|x_i^k-\bar{x}_i^k| \Big| = \Big| (v_i^k-\bar{v}_i^k) \cdot\mathrm{sgn}(x_i^k-\bar{x}_i^k) \Big| \leq |v_i^k-\bar{v}_i^k|.
\end{align*} 
Then, we use H\"older's inequality to derive
\begin{align*}
\begin{aligned}
& \Big| p|x_i-\bar{x}_i|_p^{p-1}\frac{d}{dt}|x_i-\bar{x}_i|_p \Big| = \Big| \frac{d}{dt}|x_i-\bar{x}_i|_p^p \Big| \leq p\sum_{k=1}^d |x_i^k-\bar{x}_i^k|^{p-1}|v_i^k-\bar{v}_i^k| \\
& \hspace{1cm} \leq p\left(\sum_{k=1}^d	 |x_i^k-\bar{x}_i^k|^p\right)^{\frac{p-1}{p}}\left(\sum_{k=1}^d |v_i^k-\bar{v}_i^k|^p\right)^{\frac{1}{p}} =p|x_i-\bar{x}_i|_p^{p-1}|v_i-\bar{v}_i|_p.
\end{aligned}
\end{align*}
This leads to
\[ \Big| \frac{d}{dt}|x_i-\bar{x}_i|_p \Big| \leq |v_i-\bar{v}_i|_p. \]
Again, we use H\"older's inequality, which gives
\begin{align*}
\begin{aligned}
&  \Big| q\|X-\bar{X}\|_{q,p}^{q-1}\frac{d}{dt}\|X-\bar{X}\|_{q,p}  \Big| = \Big| \frac{d}{dt}\|X-\bar{X}\|_{q,p}^q \Big|\\
& \hspace{1cm}  = \Big|q\sum_{i=1}^\infty m_i|x_i-\bar{x}_i|_p^{q-1}\frac{d}{dt}|x_i-\bar{x}_i|_p \Big| \leq q\sum_i m_i|x_i-\bar{x}_i|_p^{q-1}|v_i-\bar{v}_i|_p \\
& \hspace{1cm} \leq q\left(\sum_{i=1}^\infty m_i |x_i-\bar{x}_i|_q^q\right)^{\frac{q-1}{q}}\left(\sum_i m_i|v_i-\bar{v}_i|_p^q\right)^{\frac{1}{q}}\\
& \hspace{1cm} =q\|X-\bar{X}\|^{q-1}_{q,p}\|V-\bar{V}\|_{q,p}.
\end{aligned}
\end{align*}
Therefore, we have
\[ \Big| \frac{d}{dt}\|X-\bar{X}\|_{q,p} \Big| \leq \|V-\bar{V}\|_{q,p}.\]
In particular, for $q = p$, we have the first inequality in \eqref{C-5}. \newline

\noindent $\star$ Case B  (Derivation of the second inequality in \eqref{C-5}):~We consider $\frac{d}{dt}\|V-\bar{V}\|_p$. From \eqref{C-1}, one can directly get
	\begin{align} \label{C-7}
	\begin{aligned}
		\frac{d}{dt}(v_i^k-\bar{v}_i^k) &= \kappa\sum_{j=1}^\infty m_{j}\phi(\bw(X))\Big((v_j^k-v_i^k)-(\bar{v}_j^k-\bar{v}_i^k)\Big)\\
		&\quad+\kappa \sum_{j=1}^\infty m_j \Big(\phi(\bw(X))-\phi(\bw(\bar{X}))\Big)(\bar{v}_j^k-\bar{v}_i^k).
		\end{aligned}
	\end{align}
	We multiply by $2(v_i^k-\bar{v}_i^k)$ on both sides of \eqref{C-7} to get
	\begin{align*}
		\frac{d}{dt}|v_i^k-\bar{v}_i^k|^2 &= 2\kappa\sum_{j=1}^\infty m_{j}\phi(\bw(X)\Big((v_j^k-\bar{v}_j^k)-(v_i^k-\bar{v}_i^k)\Big)(v_i^k-\bar{v}_i^k)\\
		&\quad+2\kappa \sum_{j=1}^\infty m_{j} \Big(\phi(\bw(X))-\phi(\bw(\bar{X}))\Big)(\bar{v}_j^k-\bar{v}_i^k)(v_i^k-\bar{v}_i^k).
	\end{align*}
Then we use 
\[\frac{d}{dt}|v_i^k-\bar{v}_i^k|^p = \frac{d}{dt}(|v_i^k-\bar{v}_i^k|^2)^{\frac{p}{2}}= \frac{p}{2}|v_i^k-\bar{v}_i^k|^{p-2}\frac{d}{dt}|v_i^k-\bar{v}_i^k|^2\]
to derive
\begin{align*}
	&\frac{d}{dt}\sum_{i=1}^\infty\sum_{k=1}^d m_i |v_i^k-\bar{v}_i^k|^p=\frac{d}{dt}\|V-\bar{V}\|_p^p \\
	& \hspace{0.5cm} =p\kappa\sum_{i,j=1}^\infty\sum_{k=1}^d m_im_j\phi(\bw(X))|v_i^k-\bar{v}_i^k|^{p-2}\Big((v_j^k-\bar{v}_j^k)-(v_i^k-\bar{v}_i^k)\Big)(v_i^k-\bar{v}_i^k)\\
	& \hspace{0.7cm} +p\kappa\sum_{i,j=1}^\infty\sum_{k=1}^d m_im_j \Big(\phi(\bw(X))-\phi(\bw(\bar{X}))\Big)|v_i^k-\bar{v}_i^k|^{p-2} (\bar{v}_j^k-\bar{v}_i^k)(v_i^k-\bar{v}_i^k)\\
	& \hspace{0.5cm} =\frac{p\kappa}{2}\sum_{i,j=1}^\infty\sum_{k=1}^d m_im_j\phi(\bw(X))\Big(|v_i^k-\bar{v}_i^k|^{p-2}(v_i^k-\bar{v}_i^k)-|v_j^k-\bar{v}_j^k|^{p-2}(v_j^k-\bar{v}_j^k)\Big)\\
	&\hspace{0.7cm}\times\Big((v_j^k-\bar{v}_j^k)-(v_i^k-\bar{v}_i^k)\Big)\\
	&\hspace{0.7cm} +p\kappa\sum_{i,j=1}^\infty\sum_{k=1}^d m_im_j \Big(\phi(\bw(X))-\phi(\bw(\bar{X}))\Big)|v_i^k-\bar{v}_i^k|^{p-2}(\bar{v}_j^k-\bar{v}_i^k)(v_i^k-\bar{v}_i^k)\\
	& \hspace{0.5cm} =:\mathcal{I}_{11} + \mathcal{I}_{12}.
\end{align*}	
Below, we estimate the term ${\mathcal I}_{1i},~i = 1,2$ one by one. \newline

\noindent$\clubsuit$ (Estimate of $\mathcal{I}_{11}$): We use the relation
	\begin{align*}
		\Big(|v_i^k-\bar{v}_i^k|^{p-2}(v_i^k-\bar{v}_i^k)-|v_j^k-\bar{v}_j^k|^{p-2}(v_j^k-\bar{v}_j^k)\Big)\Big((v_j^k-\bar{v}_j^k)-(v_i^k-\bar{v}_i^k)\Big)\leq 0
	\end{align*}
	and
	\[  \bw(X)\le 2\|X\|_{p}\le 2x_p^{\infty} \]
	to find 
\begin{align*}
\begin{aligned}
		\mathcal{I}_{11} & \le\frac{p\kappa}{2}\phi(2x_p^{\infty})\sum_{i,j=1}^\infty\sum_{k=1}^d m_im_j\Big(|v_i^k-\bar{v}_i^k|^{p-2}(v_i^k-\bar{v}_i^k)-|v_j^k-\bar{v}_j^k|^{p-2}(v_j^k-\bar{v}_j^k)\Big)\\
	&\quad 	\times\Big((v_j^k-\bar{v}_j^k)-(v_i^k-\bar{v}_i^k)\Big).
	\end{aligned}
	\end{align*}
We set 
\[
a_i^k:=v_i^k-\bar v_i^k,
\]
and use Lemma \ref{L2.3} to find 
\[
\Big(|a_i^k|^{p-2}a_i^k-|a_j^k|^{p-2}a_j^k\Big)(a_i^k-a_j^k)
\ge 2^{2-p}|a_i^k-a_j^k|^p,
\qquad p\geq2.
\]
Hence, we have
\begin{align*}
	\mathcal I_{11}
	&\le
	-\frac{p\kappa}{2}\phi(2x_p^\infty)
	\sum_{i,j=1}^\infty\sum_{k=1}^d m_im_j
	\Big(|a_i^k|^{p-2}a_i^k-|a_j^k|^{p-2}a_j^k\Big)(a_i^k-a_j^k) \\
	&\le
	- p\kappa 2^{1-p}\phi(2x_p^\infty)
	\sum_{i,j=1}^\infty\sum_{k=1}^d m_im_j |a_i^k-a_j^k|^p.
\end{align*}
Moreover, we use 
\[ \sum_{j=1}^\infty m_j=1 \quad \mbox{and} \quad \sum_{j=1}^\infty m_j a_j^k=0 \]
to obtain
\[
|a_i^k|^p
=
\left|\sum_{j=1}^\infty m_j(a_i^k-a_j^k)\right|^p
\le
\sum_{j=1}^\infty m_j|a_i^k-a_j^k|^p.
\]
This yields
\[
\sum_{i=1}^\infty\sum_{k=1}^d m_i|a_i^k|^p
\le
\sum_{i,j=1}^\infty\sum_{k=1}^d m_im_j|a_i^k-a_j^k|^p.
\]
Therefore, we have
\[
\mathcal I_{11}
\le
- p\kappa 2^{1-p}\phi(2x_p^\infty)
\sum_{i=1}^\infty\sum_{k=1}^d m_i |v_i^k-\bar v_i^k|^p=- p\kappa 2^{1-p}\phi(2x_p^\infty)\|V-\bar{V}\|_p^p.
\]
\vspace{.3cm}
	
\noindent	$\clubsuit$ (Estimate of $\mathcal{I}_{12}$): We use 
\[
|\bw(X)-\bw(\bar{X})|\leq \sum_{i,j=1}^\infty m_im_j \left(|x_j-\bar{x}_j|_p+|x_i-\bar{x}_i|_p\right)
\]
and H\"older's inequality to get
	\begin{align} \label{C-8}
		\begin{aligned}
			|\cI_{12}|&\leq p\kappa\|\phi\|_{\text{Lip}}\sum_{i,j=1}^\infty m_i m_j \Bigg(\sum_{i,j=1}^\infty m_i m_j (|x_j - \bar{x}_j|_p+|x_i-\bar{x}_i|_p)\Bigg)\sum_{k=1}^d|v_i^k-\bar{v}_i^k|^{p-1}|\bar{v}_j^k-\bar{v}_i^k| \\
			&\leq p\kappa \|\phi\|_{\text{Lip}}\sum_{i,j=1}^\infty m_im_j\Bigg(\sum_{j=1}^\infty m_j |x_j - \bar{x}_j|_p+\sum_{i=1}^\infty m_i|x_i-\bar{x}_i|_p\Bigg)|v_i-\bar{v}_i|^{p-1}_p|\bar{v}_j-\bar{v}_i|_p \\
			&\leq 2p\kappa \|\phi\|_{\text{Lip}}\|X-\bar{X}\|_{1,p}\sum_{i,j=1}^\infty m_im_j|v_i-\bar{v}_i|^{p-1}_p(|\bar{v}_j|_p+|\bar{v}_i|_p) \\
			&:=2p\kappa \|\phi\|_{\mathrm{Lip}}\|X-\bar{X}\|_{1,p}(\cI_{121}+\cI_{122}).
		\end{aligned}
	\end{align}
We now estimate $\cI_{121}$ and $\cI_{122}$. \newline
	
\noindent	$\diamond$ (Estimate of $\mathcal{I}_{121}$): One can use Jensen's inequality to get
	\begin{align} \label{C-9}
		\begin{aligned}
			\cI_{121}&:=\sum_{i,j=1}^\infty  m_i m_j |v_i-\bar{v}_i|_p^{p-1}|\bar{v}_j|_p\\
			&=\Bigg(\sum_{i=1}^\infty m_i |v_i-\bar{v}_i|_p^{p-1}\Bigg)\Bigg(\sum_{j=1}^\infty m_j|\bar{v}_j|_p\Bigg)\\
			&=\|V-\bar{V}\|_{p-1,p}^{p-1}\|\bar{V}\|_{1,p}\\
			&\leq \|V-\bar{V}\|_p^{p-1}\|\bar{V}\|_{p}.
		\end{aligned}
	\end{align}
	\noindent	$\diamond$ (Estimate of $\mathcal{I}_{122}$): Again, we use H\"older's inequality to obtain 
	\begin{align} \label{C-10}
		\begin{aligned}
			\cI_{122}&:=\sum_{i,j=1}^\infty m_i m_j |v_i-\bar{v}_i|_p^{p-1}|\bar{v}_i|_p\\
			&= \sum_{i=1}^\infty m_i |v_i-\bar{v}_i|_p^{p-1}|\bar{v}_i|_p\\
			&\leq \|V-\bar{V}\|_p^{p-1}\|\bar{V}\|_p.
		\end{aligned}
	\end{align}
	We collect \eqref{C-8}, \eqref{C-9},\eqref{C-10} to find 
	\begin{align*} \label{C-13}
		\begin{aligned}
			|\cI_{12} |&\leq 4p\kappa \|\phi\|_{\mathrm{Lip}}\|V-\bar{V}\|_p^{p-1}\|\bar{V}\|_p\|X-\bar{X}\|_{1,p}\\
			&\leq 4p\kappa\|\phi\|_{\mathrm{Lip}}\|V-\bar{V}\|_p^{p-1}\|\bar{V}\|_p\|X-\bar{X}\|_p,
		\end{aligned}
	\end{align*}
and combine estimates for $\cI_{11}$ and $\cI_{12}$ to see 
	\begin{align*}\frac{d}{dt}\|V-\bar{V}\|_p^p&\le - p\kappa 2^{1-p}\phi(2x_p^\infty)\|V-\bar{V}\|_p^p+4p\kappa \|\phi\|_{\mathrm{Lip}}\|V-\bar{V}\|_p^{p-1}\|\bar{V}\|_p\|X-\bar{X}\|_{p}.
	\end{align*}
	This leads to 
	\begin{equation*}
		\frac{d}{dt}\|V-\bar{V}\|_p\leq - \kappa 2^{1-p}\phi(2x_p^\infty)\|V-\bar{V}\|_p+4\kappa \|\phi\|_{\mathrm{Lip}}\|\bar{V}\|_p\|X-\bar{X}\|_{p}.
	\end{equation*}
	Then Proposition \ref{P3.1} yields
\begin{align*}
	\frac{d}{dt}\|V-\bar{V}\|_p&\le - \kappa 2^{1-p}\phi(2x_p^\infty)\|V-\bar{V}\|_p+4\kappa \|\phi\|_{\mathrm{Lip}}e^{-\frac{\kappa \tilde{c}_p}{2} \phi(2\bar{x}_{p}^{\infty})t}\|\bar{V}^{\mathrm{in}}\|_{p}\|X-\bar{X}\|_{p}.
	\end{align*}

\vspace{0.2cm}

\noindent $\bullet$~Step B:~Finally, we set
\[\alpha=\kappa \min\left\{2^{1-p}\phi(2x_p^{\infty}),~ \frac{\tilde{c}_p}{2}\phi(2\bar{x}_p^{\infty})\right\},\quad\gamma=4\kappa \|\phi\|_{\mathrm{Lip}}\|\bar{V}^{\mathrm{in}}\|_{p} \]
	in Lemma \ref{L2.6}	to show the desired estimate:
\begin{align*}\label{C-14}
	\begin{aligned}
		\|X-\bar{X}\|_{p} +	\|V(t)-\bar{V}(t)\|_{p}\le C(\|X^{\mathrm{in}}-\bar{X}^{\mathrm{in}}\|_{p}+\|V^{\mathrm{in}}-\bar{V}^{\mathrm{in}}\|_{p}).
	\end{aligned}
\end{align*}

\section{Uniform-in-time  mean-field limit} \label{sec:4}
\setcounter{equation}{0} 
In this section, we study the uniform-in-time stability of the KCS-type model by lifting the corresponding result from the ICS-type model in Section \ref{sec:3}. For this, we first establish the uniform-in-time mean-field limit from ICS-type model to KCS-type model through particle-in-cell method, and then we derive the uniform-in-time stability for the KCS-type model by lifting the corresponding particle result. 
\subsection{Proof of Theorem \ref{T2.2}}\label{sec:4.1}
Recall the KCS-type model:
\begin{equation} 
	\begin{cases} \label{D-1}
		\displaystyle \partial_tf+ \nabla_x \cdot (v f) + \nabla_v \cdot (L[f]f)=0, \quad t > 0,~~(x, v) \in {\mathbb R}^{2d},  \vspace{6pt}\\
		\displaystyle L[f](t,x,v)=-\kappa\int_{\mathbb{R}^{2d}} \phi( \bw[f](t)) \left(v-v_*\right)  f_*d x_* d v_*,\vspace{6pt} \\
		\displaystyle \bw[f](t)=\int_{\mathbb{R}^{4d}}|x-x_*|_pf f_*d x d vd x_* d v_*,\vspace{6pt} \\
		\displaystyle f \Big|_{t = 0}  =f^{\mathrm{in}}.
	\end{cases}
\end{equation}
Let $\cP(\bbr^{2d})$ be the set of all probability measures on $\bbr^{2d}$. First, we recall the definition of  a measure-valued solution to \eqref{D-1}. 
\begin{definition} \label{D4.1}
	For $T\in[0,\infty)$, $\mu_t\in L^\infty ([0,T);\cP(\bbr^{2d}))$ is a measure-valued solution to \eqref{D-1} with initial data $\mu_0$ if the following conditions hold:
	\begin{enumerate}
		\item The map $t~\mapsto~\mu_t$ is weakly continuous in the sense that 
			\[\int_{\bbr^{2d}} \varphi(x,v) \mu(t,dx,dv)\quad \text{is continuous in }t,\quad \forall~\varphi \in \cC_c^1(\bbr^{2d}). \]
		\item $\mu$ satisfies the equation \eqref{D-1} in a weak sense: for any $\varphi\in \cC_c^1([0,T)\times \bbr^{2d})$, 
		\begin{align*}
		\begin{aligned}
		& \int_{\bbr^{2d}}\varphi (t,x,v)\mu_t(dx,dv) \\
		& \hspace{0.5cm} = \int_{\bbr^{2d}}\varphi (0,x,v)\mu_0(dx,dv) + \int_0^t \int_{\bbr^{2d}}(\partial_s\varphi +v\cdot \nabla_x \varphi +L[\mu_s]\cdot \nabla_v\varphi)\mu_s(dx,dv)~	ds.
		\end{aligned}
		\end{align*}
	\end{enumerate}
\end{definition}
\begin{remark}\label{R4.1}
	Note that for a solution $(X(t), V(t))$ to \eqref{C-1} with initial data $(X^\mathrm{in},V^{\mathrm{in}})$, the measure 
	\[\mu_t:=\sum_{i=1}^\infty m_i \delta_{x_i(t)}\otimes \delta_{v_i(t)}\]
	is a measure-valued solution of \eqref{D-1} with initial datum:
	\[\mu_0:=\sum_{i=1}^\infty m_i\delta_{x_i^{\mathrm{in}}}\otimes \delta_{v_i^\mathrm{in}},\]
	in the sense of Definition \ref{D4.1}.
\end{remark}
\vspace{0.2cm}

Recall the notation
\begin{align*}
	\|X_{\mu}\|_{q,p}^q:=	\int_{\bbr^{2d}}|x|_p^q~ \mu(dx,dv), \quad 	\|V_{\mu}\|_{q,p}^q:=	\int_{\bbr^{2d}}|v|_p^q~ \mu(dx,dv).
\end{align*}
We show that the solution of \eqref{D-1} can be approximated by the countable sum of Dirac measures using the infinite particle system. Before moving further, we summarize our proof strategy below.
\begin{itemize}
	\item Step A: We use particle-in-cell method to construct approximation sequence $\mu_0^n$ such that for any $p,q\in[1,\infty]$,
	\[W_{q,p}(\mu_0, \mu_0^n)\to 0 \quad \text{as}\quad n\to\infty.\]
	\item Step B: Next, we estimate the distance between approximation measures.
	\item Step C: Finally, we show that for any $t>0$, $\{\mu_t^n\}_{n=1}^\infty$ is a Cauchy sequence, which yields the uniform-in-time mean-field limit results.
\end{itemize}
Now, we proceed with our proof step by step below.\newline

	\noindent $\bullet$ \textbf{Step A (Construction of an approximation sequence):} For given $n\in\bbn$ and lattice points $\alpha,\beta\in\bbz^d$, we set the cube region $\Omega^n_{\alpha \beta}$ and its measure $m_{\alpha\beta}^n$ as 
\begin{equation*}
	\Omega^n_{\alpha\beta}:=\prod_{k=1}^{d}\left[\frac{\alpha_k}{2^n}, \frac{\alpha_k+1}{2^n}\right]\times \prod_{k=1}^{d}\left[\frac{\beta_k}{2^n}, \frac{\beta_k+1}{2^n}\right],\qquad m_{\alpha\beta}^n:=\mu_{0}(\Omega^n_{\alpha\beta}).
\end{equation*}
	If $m_{\alpha\beta}^n>0$, we define the cell average $(x^n_{\alpha\beta}, v^n_{\alpha\beta})\in\Omega^n_{\alpha\beta}$ as 
\begin{equation} \label{D-1-1}
	x_{\alpha\beta}^n := \frac{1}{m_{\alpha\beta}^n}\int_{\Omega^n_{\alpha\beta}}x\mu_0(dx,dv),\qquad v_{\alpha\beta}^n = \frac{1}{m_{\alpha\beta}^n}\int_{\Omega^n_{\alpha\beta}}v\mu_0(dx,dv).
\end{equation}
If $m_{\alpha\beta}^n=0$, we arbitrarily choose $(x^n_{\alpha\beta}, v^n_{\alpha\beta})\in\Omega^n_{\alpha\beta}$.\newline

Next, we define the approximate measure $\mu_0^n$ of $\mu_0$ by
\[\mu_0^n:=\sum_{\alpha,\beta\in\bbz^d}m_{\alpha\beta}^n\delta_{x_{\alpha\beta}^n}\otimes \delta_{v_{\alpha\beta}^n}.\]
Now, we show that for any $p,q\in[1,\infty]$,
\[W_{q,p}(\mu_0, \mu_0^n)\to 0 \quad \text{as}\quad n\to\infty.\]
For each cube $\Omega_{\alpha\beta}^n$, define the piecewise constant map $c_n:\bbr^{2d}\to\bbr^{2d}$ as 
\begin{align*}
	c_n(x,v):=(x_{\alpha\beta}^n, v_{\alpha\beta}^n)\quad \text{for}\quad (x,v)\in\Omega_{\alpha\beta}^n.
\end{align*}
We set
\begin{align*}
	\pi_n:=(\mathrm{Id},c_n)_\#\mu_0\in \Pi (\mu_0, \mu_0^n).
\end{align*}
In fact, the first marginal of $\pi_n$ is clearly $\mu_0$, and the second marginal of $\pi_n$ is $\mu_0^n$: for any measurable set $A\in\bbr^{2d}$, we have
\begin{equation*}
	(c_n)_\# \mu_0(A) = \sum_{\alpha,\beta\in\bbz^d} \mu_0(\Omega^n_{\alpha\beta}\cap c_n^{-1}(A)) = \sum_{\alpha, \beta\in\bbz^d} \mu_0(\Omega^n_{\alpha\beta})\mathbf{1}_A(x_{\alpha\beta}^n,v_{\alpha\beta}^n)=\mu_0^n(A).
\end{equation*}
Note that for any $(x,v)\in\Omega_{\alpha\beta}^n$, we have
\begin{align*}
	|(x,v)-(x_{\alpha\beta}^n,v_{\alpha\beta}^n)|_p\leq \left(\sum_{k=1}^{2d}\left(\frac{1}{2^{n}}\right)^p\right)^{1/p}=\frac{(2d)^{1/p}}{2^{n}}.
\end{align*}
Then for $p, q\in[1,\infty)$, we have
\begin{align*}
	W_{q,p}^q(\mu_0,\mu_0^n)&\leq \int_{\bbr^{2d}\times\bbr^{2d}}|z-z^*|_p^q~\pi_n(dz,dz^*)\\
	&=\sum_{\alpha,\beta\in\bbz^d}\int_{\bbr^{2d}\times \bbr^{2d}}|z-z^*|_p^q~(\mathrm{Id}, c_n)_\#(\mu_0|_{\Omega_{\alpha\beta}^n}) (dz,dz^*)\\
	&=\sum_{\alpha,\beta\in\bbz^d}\int_{\Omega_{\alpha\beta}^n}|(x,v)-(x_{\alpha\beta}^n,v_{\alpha\beta}^n)|_p^q~\mu_0 (dx,dv)\\
	&\leq \sum_{\alpha,\beta\in\bbz^d}\int_{\Omega_{\alpha\beta}^n}\frac{(2d)^{q/p}}{2^{nq}}~\mu_0(dx,dv)=\frac{(2d)^{q/p}}{2^{nq}}\to 0\quad \text{as} \quad n\to \infty.
\end{align*}
For $q=\infty$, we have
\[
W_{\infty,p}(\mu_0,\mu_0^n) \leq \sup_{(z,z^*)\in \operatorname{supp}\pi_n}|z-z^*|_p \leq \frac{(2d)^{1/p}}{2^{n}}\to 0\quad \text{as}\quad n\to \infty.
\]
The case for $p=\infty$ can be treated in a similar manner.

\vspace{.3cm}

	\noindent $\bullet$ \textbf{Step B (Estimation of approximation sequence):} In this step, we show the distance of different measures in our approximation sequences. First, for fixed $n\in\bbn$, we use rearrangement to express $\mu_0^n$:
\[\mu_0^n:=\sum_{\alpha,\beta\in\bbz^d}m_{\alpha\beta}^n\delta_{x_{\alpha\beta}^n}\otimes \delta_{v_{\alpha\beta}^n},\quad\text{where}\quad m_{\alpha\beta}^n:=\mu_{0}(\Omega^n_{\alpha\beta}),\]
in the form of
\begin{equation*}
	\mu_0^n=\sum_{i=1}^\infty m_i\delta_{x_i^{\mathrm{in}}}\otimes \delta_{v_i^{\mathrm{in}}},
\end{equation*}
where $(X^{\mathrm{in}}, V^{\mathrm{in}})\in \ell^{p,p}\times \ell^{p,p}$. We proceed as follows: we choose a bijection between $\bbz^d\times \bbz^d$ and $\bbn$ as $(\alpha,\beta)\leftrightarrow i$ and set
\begin{equation} \label{D-2}
m_i:= m_{\alpha\beta}^{n},\quad (x_{i}^{\mathrm{in}},v_{i}^{\mathrm{in}}):= (x_{\alpha\beta}^{n},v_{\alpha\beta}^{n}).
\end{equation}
Then we have
\begin{equation} \label{D-3}
	\mu_0^n=\sum_{i=1}^\infty m_i\delta_{x_i^{\mathrm{in}}}\otimes \delta_{v_i^{\mathrm{in}}},
\end{equation}
with $X^{\mathrm{in}}\in\ell^{p,p}$ since
\begin{equation} \label{D-3-1}
	|x_{\alpha\beta}^{n}|_p\leq |x|_p+\frac{d^{\frac{1}{p}}}{2^{n}}\quad \text{for} \quad x\in \Omega_{\alpha\beta}^{n},
\end{equation}
and
\begin{align*}
\sum_{i=1}^\infty m_i |x_i^{\mathrm{in}}|_p^p&= \sum_{\alpha, \beta\in\bbz^d}\int_{\Omega_{\alpha\beta}^{n}} |x^{n}_{\alpha\beta}|_p^p~\mu_0(dx,dv)\\
&\leq 2^{p-1}\sum_{\alpha, \beta\in\bbz^d} \int_{\Omega_{\alpha\beta}^{n}}\left(|x|_p^p+\frac{d}{2^{np}}\right)\mu_0(dx,dv)\\
&=2^{p-1}\bigg[\int_{\bbr^{2d}}|x|_p^p~\mu_0(dx,dv)+\frac{d}{2^{np}}\bigg]<\infty.
\end{align*}
Note that since \eqref{D-3-1} holds, by dominated convergence theorem, we have
\begin{equation*}
	\|X^{\mathrm{in}}\|_p=\|X_{\mu_0^n}\|_p = \bigg(\sum_{\alpha,\beta\in\bbz^d}\int_{\Omega^n_{\alpha\beta}}|x_{\alpha\beta}^n|^p \mu_0(dx,dv) \bigg)^{1/p}\to \|X_{\mu_0}\|_p\quad \text{as }n\to\infty.
\end{equation*}
Similarly, 
\[ V^{\mathrm{in}}\in \ell^{p,p} \quad \mbox{and} \quad \|V^{\mathrm{in}}\|_p=\|V_{\mu_0^n}\|_p\to \|V_{\mu_0}\|_p \quad \mbox{as $n\to\infty$}. \] 
 Furthermore, the relations \eqref{D-1-1} give
\begin{equation} \label{D-3-2}
\sum_{i=1}^\infty m_iv_i^{\mathrm{in}}=\sum_{\alpha,\beta\in\bbz^d} m^n_{\alpha\beta}v^n_{\alpha\beta} = \sum_{\alpha, \beta\in\bbz^d} \int_{\Omega^n_{\alpha\beta}}v\mu_0(dx,dv)= \int_{\bbr^{2d}} v\mu_0(dx,dv)=0.
\end{equation}
Now, we fix $n_1>n_2$ and construct
\[m_{\alpha\beta}^{n_1}=\mu_{0}(\Omega^{n_1}_{\alpha\beta}),\quad\alpha,\beta\in\bbz^d,\quad \mu_0^{n_1}:=\sum_{\alpha,\beta\in\bbz^d}m_{\alpha\beta}^{n_1}\delta_{x_{\alpha\beta}^{n_1}}\otimes \delta_{v_{\alpha\beta}^{n_1}},\]
\[m_{\tilde{\alpha}\tilde{\beta}}^{n_2}=\mu_{0}(\Omega^{n_2}_{\tilde{\alpha}\tilde{\beta}}),\quad \tilde{\alpha},\tilde{\beta}\in\bbz^d,\quad \mu_0^{n_2}:=\sum_{\tilde{\alpha},\tilde{\beta}\in\bbz^d}m_{\tilde{\alpha}\tilde{\beta}}^{n_2}\delta_{x_{\tilde{\alpha}\tilde{\beta}}^{n_2}}\otimes \delta_{v_{\tilde{\alpha}\tilde{\beta}}^{n_2}}.\]
Since $n_1>n_2$, each large cube $\Omega_{\tilde{\alpha}\tilde{\beta}}^{n_2}$ is a disjoint union of $2^{2d(n_1-n_2)}$ number of $\Omega^{n_1}_{\alpha\beta}$'s: 
\[\Omega_{\tilde{\alpha} \tilde{\beta}}^{n_2} = \bigcup_{(\alpha,\beta)\in S(\tilde{\alpha},\tilde{\beta})} \Omega_{\alpha\beta}^{n_1},\quad \text{where}\quad |S(\tilde{\alpha},\tilde{\beta})| = 2^{2d(n_1-n_2)}, \]
and therefore
\[m_{\tilde{\alpha}\tilde{\beta}}^{n_2} = \sum_{(\alpha,\beta)\in S(\tilde{\alpha},\tilde{\beta})} m_{\alpha\beta}^{n_1}.\]
We use \eqref{D-2} and \eqref{D-3} with $n=n_1$ so that we use a bijection between $\bbz^d\times \bbz^d$ and $\bbn$ with $(\alpha,\beta)\longleftrightarrow i$ to express $\mu_0^{n_1}$ as
\begin{equation} \label{D-4}
	\mu_0^{n_1}=\sum_{i=1}^\infty m_i\delta_{x_i^{\mathrm{in}}}\otimes \delta_{v_i^{\mathrm{in}}}.
\end{equation}
Now, we define
\begin{equation} \label{D-5}
(\tilde{\alpha},\tilde{\beta}) := \left(\bigg\lfloor\frac{\alpha}{2^{n_1-n_2}}\bigg\rfloor, \bigg\lfloor\frac{\beta}{2^{n_1-n_2}}\bigg\rfloor\right),\quad (\tilde{x}_{i}^{\mathrm{in}},\tilde{v}_{i}^{\mathrm{in}}):=(x_{\tilde{\alpha}\tilde{\beta}}^{n_2},v_{\tilde{\alpha}\tilde{\beta}}^{n_2}).
\end{equation}
Note that the same $(\tilde{\alpha},\tilde{\beta})$ and $(\tilde{x}_i^{\mathrm{in}}, \tilde{v}_i^{\mathrm{in}})$ is repeated for $2^{2d(n_1-n_2)}$ number of different $(\alpha,\beta)$'s. Then, we have 
\[ (X^{\mathrm{in}}, V^{\mathrm{in}}),(\tilde{X}^{\mathrm{in}}, \tilde{V}^{\mathrm{in}})\in \ell^{p,p}\times \ell^{p,p}, \]
where 
\[\|X^{\mathrm{in}}\|_p = \|X_{\mu_0^{n_1}}\|_p, \quad\|V^{\mathrm{in}}\|_p = \|V_{\mu_0^{n_1}}\|_p,\quad \|\tilde{X}^{\mathrm{in}}\|_p = \|X_{\mu_0^{n_2}}\|_p, \quad\|\tilde{V}^{\mathrm{in}}\|_p = \|V_{\mu_0^{n_2}}\|_p.\]One should keep in mind that we are using the weights $\{m_i\}_{i=1}^\infty =\{m_{\alpha\beta}^{n_1}\}_{\alpha,\beta\in\bbz^d}$ when discussing the space $\ell^{p,p}$. For instance, $\tilde{X}^{\mathrm{in}}\in \ell^{p,p}$ comes from
\begin{equation*}
\|\tilde{X}^{\mathrm{in}}\|_p^p
=\sum_{i=1}^\infty m_i |\tilde{x}_i^{\mathrm{in}}|_p^p 
= \sum_{\tilde{\alpha},\tilde{\beta}\in\bbz^d}\left(\sum_{(\alpha,\beta)\in S (\tilde{\alpha},\tilde{\beta})} m_{\alpha\beta}^{n_1}\right)|x_{\tilde{\alpha}\tilde{\beta}}^{n_2}|_p^p = \sum_{\tilde{\alpha},\tilde{\beta}\in\bbz^d} m_{\tilde{\alpha}\tilde{\beta}}^{n_2}|x^{n_2}_{\tilde{\alpha}\tilde{\beta}}|_p^p = \|X_{\mu_0^{n_2}}\|^p_p.
\end{equation*}
Then, since
\[\|V_{\mu_0}\|_{p}< \frac{\kappa \tilde{c}_p}{4} \int_{2\|X_{\mu_0}\|_{p}}^{\infty}\phi(s)d s\]
and
\[\|V_{\mu_0^{n}}\|_{p}\to \|V_{\mu_0}\|_{p},\quad \|X_{\mu_0^{n}}\|_{p}\to \|X_{\mu_0}\|_{p} \quad \mbox{as $n \to \infty$}, \]
for sufficiently large $n_1,n_2$, we have 
\begin{equation} \label{D-5-1}
	\|V^{\mathrm{in}}\|_{p}< \frac{\kappa \tilde{c}_p}{4} \int_{2\|X^{\mathrm{in}}\|_{p}}^{\infty}\phi(s)d s,\quad\| \tilde{V}^{\mathrm{in}}\|_{p}< \frac{\kappa \tilde{c}_p}{4} \int_{2\|\tilde{X}^{\mathrm{in}}\|_{p}}^{\infty}\phi(s)d s.
\end{equation}
Moreover, by \eqref{D-3-2}, we have
\begin{equation} \label{D-5-2}
	\sum_{i=1}^\infty m_i v_i^{\mathrm{in}}=0, \quad \sum_{i=1}^\infty m_i \tilde{v}_i^{\mathrm{in}}
	=\sum_{\tilde{\alpha},\tilde{\beta}\in\bbz^d}\left(\sum_{(\alpha,\beta)\in S (\tilde{\alpha},\tilde{\beta})} m_{\alpha\beta}^{n_1}\right)v_{\tilde{\alpha}\tilde{\beta} }^{n_2} = \sum_{\tilde{\alpha}, \tilde{\beta}\in\bbz^d} m_{\tilde{\alpha}\tilde{\beta}}^{n_2}v_{\tilde{\alpha}\tilde{\beta} }^{n_2}=0.
\end{equation}
On the other hand, since
\[
\sum_{i=1}^\infty m_i \delta_{\tilde{x}_{i}^{\mathrm{in}}}\otimes\delta_{\tilde{v}_{i}^{\mathrm{in}}} 
= \sum_{\tilde{\alpha},\tilde{\beta}\in\bbz^d}\left(\sum_{(\alpha,\beta)\in S(\tilde{\alpha},\tilde{\beta})}m_{\alpha\beta}^{n_1}\right)\delta_{x_{\tilde{\alpha}\tilde{\beta}}^{n_2}}\otimes\delta_{v_{\tilde{\alpha}\tilde{\beta}}^{n_2}} = \sum_{\tilde{\alpha},\tilde{\beta}\in\bbz^d}m_{\tilde{\alpha}\tilde{\beta}}^{n_2}\delta_{x_{\tilde{\alpha}\tilde{\beta}}^{n_2}}\otimes\delta_{v_{\tilde{\alpha}\tilde{\beta}}^{n_2}} = \mu_0^{n_2},
\]
$\mu_0^{n_2}$ can also be expressed as
\begin{equation} \label{D-6}
\mu_0^{n_2} = \sum_{i=1}^\infty m_i \delta_{\tilde{x}_{i}^{\mathrm{in}}}\otimes \delta_{\tilde{v}_{i}^{\mathrm{in}}}.
\end{equation}
Moreover, the relation \eqref{D-5} implies that for each $i\in\bbn$, $(x_{i}^{\mathrm{in}},v_{i}^{\mathrm{in}})$ and $(\tilde{x}_{i}^{\mathrm{in}},\tilde{v}_{i}^{\mathrm{in}})$ are inside the same cube with length $1/2^{n_2}$:
\begin{equation*} 
(x_{i}^{\mathrm{in}},v_{i}^{\mathrm{in}}) \in \Omega_{\tilde{\alpha}\tilde{\beta}}^{n_2}, \quad (\tilde{x}_{i}^{\mathrm{in}},\tilde{v}_{i}^{\mathrm{in}}) \in \Omega_{\tilde{\alpha}\tilde{\beta}}^{n_2},	
\end{equation*}
where $(\tilde{\alpha}, \tilde{\beta})$ is defined by \eqref{D-5} with $(\alpha, \beta)\leftrightarrow i$. This gives
\begin{equation} \label{D-7}
	\|X^{\mathrm{in}}-\tilde{X}^{\mathrm{in}}\|_{p}\leq \left(\sum_{k=1}^d\left(\frac{1}{2^{n_2}}\right)^p\right)^{1/p}=\frac{d^{1/p}}{2^{n_2}}\quad \text{and}\quad \|V^{\mathrm{in}}-\tilde{V}^{\mathrm{in}}\|_{p}\leq\frac{d^{1/p}}{2^{n_2}}.
\end{equation}
\vspace{.3cm}

\noindent $\bullet$ \textbf{Step C (Derivation of uniform mean-field limit).} Next, we show that for any $t>0$, $\{\mu_t^n\}_{n=1}^\infty$ is a Cauchy sequence. Let $(X(t), V(t)),(\tilde{X}(t), \tilde{V}(t))\in \ell^{p,p}\times \ell^{p,p}$ be global solutions to \eqref{C-1} with initial data $(X^{\mathrm{in}}, V^{\mathrm{in}})$ and $(\tilde{X}^{\mathrm{in}}, \tilde{V}^{\mathrm{in}})$, respectively. \newline

It follows from \eqref{D-4} and \eqref{D-6} that we can choose the transport measure $\pi_t\in \mathcal{P}(\bbr^{2d}\times\bbr^{2d})$ of $\mu_t^{n_1}$ and $\mu_t^{n_2}$ as 
\begin{equation*}
	\pi_t := \sum_{i=1}^\infty m_i \delta_{x_i(t)}\otimes \delta_{v_i(t)}\otimes \delta_{\tilde{x}_i(t)}\otimes \delta_{\tilde{v}_i(t)}.
\end{equation*}
Then we have
\begin{align*}
	W_p^p(\mu_t^{n_1},\mu_t^{n_2}) &= W_p^p\left(\sum_{i
	=1}^\infty m_i \delta_{x_{i}(t)}\otimes \delta_{v_{i}(t)},\sum_{i=1}^\infty m_i \delta_{\tilde{x}_{i}(t)}\otimes \delta_{\tilde{v}_{i}(t)}\right)\\
	&\leq \int_{\bbr^{2d}\times \bbr^{2d}}|(x,v)-(\tilde{x},\tilde{v})|_p^p~d\pi_t(x,v,\tilde{x},\tilde{v})\\
	&=\sum_{i=1}^\infty m_i(|x_i(t)-\tilde{x}_i(t)|_p^p+|v_i(t)-\tilde{v}_i(t)|_p^p)\\
	&=\|X(t)-\tilde{X}(t)\|_p^p+\|V(t)-\tilde{V}(t)\|_p^p.
\end{align*}
Now, since \eqref{D-5-1} and \eqref{D-5-2} hold, we apply Theorem \ref{T2.1} to find 
\begin{equation} \label{D-7-1}
	W_p(\mu_t^{n_1},\mu_t^{n_2}) \leq \|X(t)-\tilde{X}(t)\|_p+\|V(t)-\tilde{V}(t)\|_p \leq C(\|X^{\mathrm{in}}-\tilde{X}^{\mathrm{in}}\|_{p}+ \|V^{\mathrm{in}}-\tilde{V}^{\mathrm{in}}\|_{p}).
\end{equation}
Since
\[\|X_{\mu_0^{n}}\|_{p}\to \|X_{\mu_0}\|_{p} \quad \text{and}\quad  \|V_{\mu_0^{n}}\|_{p}\to \|V_{\mu_0}\|_{p}\quad \mbox{as $n \to \infty$},\]
the quantities $\|X_{\mu_0^n}\|_p$ and $\|V_{\mu_0^n}\|_p$ are uniformly bounded for all sufficiently large $n$. Moreover, as in Step B of Subsection \ref{sec:3.2}, the constant $C>0$ in \eqref{D-7-1} only depends on the uniform bound of $\|X_{\mu_0^n}\|_{p}, \|V_{\mu_0^n}\|_{p}$. Hence, for all sufficiently large $n_1$ and $n_2$, we can choose $C>0$ in \eqref{D-7-1} independently of $n_1$ and $n_2$. Then, we use \eqref{D-7} to find
\[W_{p}(\mu_t^{n_1},\mu_t^{n_2})\leq \frac{Cd^{1/p}}{2^{n_2-1}}.\]
Therefore, for any $\e>0$, there exists $N\in\bbn$ such that for $n_1,n_2\geq N$, we have
\begin{equation} \label{D-8}
	 W_p(\mu_t^{n_1},\mu_t^{n_2})\leq \e,\quad \forall~t\geq0.
\end{equation}
This implies that for each $t\geq 0$, $\{\mu_t^n\}_{n=1}^\infty$ is a Cauchy sequence in $W_p$-metric, and there exists $\mu_t\in \mathcal P_p(\mathbb R^{2d})$ such that
\[W_p(\mu_t^n,\mu_t)\to 0 \qquad \mbox{as } n\to\infty.\]
Moreover, the relation \eqref{D-8} gives that if $n\geq N$, then
\[\sup_{t\in[0,\infty)}W_p(\mu_t^n, \mu_t)\leq \e.\]
Since $\e>0$ was arbitrary, we have
\[\lim_{n\to\infty} \sup_{t\in[0,\infty)}W_p(\mu_t^n,\mu_t)=0.\]

It remains to verify that the limit $\mu_t$ is a measure-valued solution
to \eqref{D-1}. Let $z=(x,v)$, and recall
\[
\bw[\mu]:=\int_{\bbr^{2d}\times\bbr^{2d}}
|x-x'|_p\,\mu(dz)\mu(dz'),\quad L[\mu](z):=-\kappa\phi(\bw[\mu])\left(v-\int_{\bbr^{2d}}v_*\mu(dz_*)\right).
\]
For each $s\geq 0$, let $\pi_s^n\in\Pi(\mu_s^n,\mu_s)$ be an optimal coupling
for $W_p$. Then we have
\[
\begin{aligned}
	|\bw[\mu_s^n]-\bw[\mu_s]|
	&\leq
	\int_{\bbr^{4d}}\int_{\bbr^{4d}}
	\Bigl|
	|x-x'|_p-|\bar x-\bar x'|_p
	\Bigr|
	\,\pi_s^n(dz,d\bar z)\pi_s^n(dz',d\bar z')  \\
	&\leq
	\int_{\bbr^{4d}}\int_{\bbr^{4d}}
	\bigl(|x-\bar x|_p+|x'-\bar x'|_p\bigr)
	\,\pi_s^n(dz,d\bar z)\pi_s^n(dz',d\bar z')  \\
	&=
	2\int_{\bbr^{4d}} |x-\bar x|_p\,\pi_s^n(dz,d\bar z)
	\leq 2W_p(\mu_s^n,\mu_s).
\end{aligned}
\]
Moreover, since 
\[\int_{\bbr^{2d}}v_*\mu_s^n(dz_*)=\int_{\bbr^{2d}}v_*\mu_s(dz_*)=0,\]
for every compact set $K\subset\mathbb R^{2d}$,
\[
\sup_{z\in K}|L[\mu_s^n](z)-L[\mu_s](z)|
\leq C_K W_p(\mu_s^n,\mu_s),
\]
where $C_K>0$ is independent of $n$ and $s$.

Now let $\varphi\in \cC_c^1([0,T)\times \bbr^{2d})$. Since each $\mu_t^n$ is a
measure-valued solution, we have
\begin{align*}
	\begin{aligned}
		& \int_{\bbr^{2d}}\varphi (t,z)\mu_t^n(dz) \\
		& \hspace{0.5cm} = \int_{\bbr^{2d}}\varphi (0,z)\mu_0^n(dz) + \int_0^t \int_{\bbr^{2d}}(\partial_s\varphi +v\cdot \nabla_x \varphi +L[\mu^n_s]\cdot \nabla_v\varphi)\mu^n_s(dz)~	ds.
	\end{aligned}
\end{align*}
Using the convergence $\mu_s^n\to\mu_s$ in $W_p$, the above estimate for
$L[\mu_s^n]-L[\mu_s]$, and the compact support of $\varphi$, we can pass to the limit
$n\to\infty$ in the above weak formulation. Thus,
\begin{align*}
	\begin{aligned}
		& \int_{\bbr^{2d}}\varphi (t,z)\mu_t(dz) \\
		& \hspace{0.5cm} = \int_{\bbr^{2d}}\varphi (0,z)\mu_0(dz) + \int_0^t \int_{\bbr^{2d}}(\partial_s\varphi +v\cdot \nabla_x \varphi +L[\mu_s]\cdot \nabla_v\varphi)\mu_s(dz)~	ds.
	\end{aligned}
\end{align*}
Moreover, the weak continuity of $t\mapsto\mu_t$ follows from $\sup_{t\in[0,\infty)}W_p(\mu_t^n,\mu_t)\to0$ and the weak continuity of
$t\mapsto\mu_t^n$. Hence $\mu_t$ is a measure-valued solution to \eqref{D-1}
in the sense of Definition \ref{D4.1}.

This completes the proof of Theorem \ref{T2.2}.
\subsection{Proof of Theorem \ref{T2.3}}\label{sec:4.2}
Before we discuss the final result on the uniform-in-time stability of the KCS-type model, we recall Wasserstein-type distance $W_{p}$ as follows:
\begin{equation} \label{D-9}
	W_{p}(\mu,\nu)=\inf\limits_{\pi\in\Pi(\mu,\nu)}\left\{\int_{\bbr^{4d}} \Big( |x-\tilde{x}|_p^p+|v-\tilde{v}|_p^p\Big) d\pi(x,v,\tilde{x},\tilde{v})\right\}^{\frac{1}{p}}.
\end{equation}
Next, we present the proof of the main theorem dealing with the uniform-in-time stability result of \eqref{D-1}. We split the proof into two steps.\newline

\noindent $\bullet$ \textbf{Step A (Construction approximation sequences of two measures):} Let $\e>0$ be given. Using Theorem \ref{T2.2} and \eqref{D-3}, we first approximate the initial data $\mu_0$ and $\nu_0$: there exists large $n\in\bbn$ such that the solutions $\mu_t^n$, $\nu_t^n$ of \eqref{D-1} with initial data
\[\mu_0^{n}:=\sum_{i=1}^\infty m^n_i\delta_{x_{i}^{n,\mathrm{in}}}\otimes \delta_{v_{i}^{n,\mathrm{in}}},\quad \nu_0^{n}:=\sum_{i=1}^\infty \tilde{m}_i^n \delta_{\tilde{x}_{i}^{n,\mathrm{in}}}\otimes \delta_{\tilde{v}_{i}^{n,\mathrm{in}}}\]
satisfy
\begin{equation*}
\sup_{t\in[0,\infty)}	W_p(\mu_t, \mu_t^n)<\frac{\e}{2},\quad \sup_{t\in[0,\infty)}W_p(\nu_t, \nu_t^n)<\frac{\e}{2},
\end{equation*}
with $(X^{n,\mathrm{in}}, V^{n,\mathrm{in}}), (\tilde{X}^{n,\mathrm{in}}, \tilde{V}^{n,\mathrm{in}})\in \ell^{p,p}\times  \ell^{p,p}$.
Then by triangle inequality, we have 
\begin{align} \label{D-10}
\begin{aligned}
	W_p(\mu_t,\nu_t)\leq W_p(\mu_t^n,\nu_t^n)+\e,\quad \forall~t>0.
	\end{aligned}
\end{align}
Now, using \eqref{D-9}, we can find an optimal plan $(m_{\ell,k})_{\ell,k\in\bbn}$ whose entries are nonnegative real numbers and it satisfies 
\begin{equation}\label{D-11}
\sum_{\ell=1}^\infty\sum_{k=1}^\infty m_{\ell,k}|x_{\ell}^{n,\mathrm{in}}-\tilde{x}_{k}^{n,\mathrm{in}}|_p^{p}+\sum_{\ell=1}^\infty\sum_{k=1}^\infty m_{\ell,k}|v_{\ell}^{n,\mathrm{in}}-\tilde{v}_{k}^{n,\mathrm{in}}|_{p}^{p}=W_{p}^{p}(\mu_0^n, \nu_0^n),
\end{equation}
where 
\begin{equation*}
	\sum_{k=1}^\infty m_{\ell,k}=m^n_\ell,	\quad \sum_{\ell=1}^\infty m_{\ell,k}=\tilde{m}^n_k.
\end{equation*}
Choose a bijection $\sigma: \bbn\to\bbn^2$ with $\sigma(j):=(\ell_j, k_j)$. For $m_{\ell_j,k_j}>0$, we define 
\begin{equation} \label{D-12}
	\mathbf{m}^n_j:=m_{\ell_j,k_j},\quad y_{j}^{n,\mathrm{in}}:=x_{\ell_j}^{n,\mathrm{in}},\quad \tilde{y}_{j}^{n,\mathrm{in}}:=\tilde{x}_{k_j}^{n,\mathrm{in}},\quad u_{j}^{n,\mathrm{in}}:=v_{\ell_j}^{n,\mathrm{in}},\quad \tilde{u}_{j}^{n,\mathrm{in}}:=\tilde{v}_{k_j}^{n,\mathrm{in}}.
\end{equation}
Then we have 
\[ (Y^{n,\mathrm{in}}, U^{n,\mathrm{in}}), (\tilde{Y}^{n,\mathrm{in}}, \tilde{U}^{n,\mathrm{in}})\in \ell^{p,p}\times \ell^{p,p}, \]
where we use the weights $\{\mathbf{m}^n_j\}_{j=1}^\infty$. For instance, we have 
\[
	\sum_{j=1}^\infty \mathbf{m}^n_j|y_j^{n,\mathrm{in}}|_p^p =\sum_{\ell, k=1}^\infty m_{\ell,k}|x_\ell^{n,\mathrm{in}}|_p^p =\sum_{\ell=1}^\infty m^n_\ell |x_\ell^{n,\mathrm{in}}|_p^p<\infty.
\]
Note that the relations \eqref{D-11} and \eqref{D-12} yield
\begin{align} \label{D-13}
	\begin{aligned}
	&\|Y^{n,\mathrm{in}}-\tilde{Y}^{n,\mathrm{in}}\|_{p}^{p}+ \|U^{n,\mathrm{in}}-\tilde{U}^{n,\mathrm{in}}\|_{p}^{p}\\
	& \hspace{1cm} =\sum_{i=1}^\infty \textbf{m}^n_{i}|y_{i}^{n,\mathrm{in}}-\tilde{y}_{i}^{n,\mathrm{in}}|_p^{p}+\sum_{i=1}^\infty \textbf{m}^n_{i} |u_{i}^{n,\mathrm{in}}-\tilde{u}_{i}^{n,\mathrm{in}}|_{p}^{p}\\
	& \hspace{1cm} =\sum_{\ell=1}^\infty\sum_{k=1}^\infty m_{\ell,k}|x_{\ell}^{n,\mathrm{in}}-\tilde{x}_{k}^{n,\mathrm{in}}|_p^{p} +\sum_{\ell=1}^\infty\sum_{k=1}^\infty m_{\ell,k}|v_{\ell}^{n,\mathrm{in}}-\tilde{v}_{k}^{n,\mathrm{in}}|_{p}^{p}=W_{p}^{p}(\mu_0^n, \nu_0^n).
	\end{aligned}
\end{align}
Now, let $(Y^n(t), U^n(t))$ and $(\tilde{Y}^n(t), \tilde{U}^n(t))\in \ell^{p,p}\times \ell^{p,p}$ be a global solution to \eqref{C-1} with initial data $(Y^{n,\mathrm{in}}, U^{n,\mathrm{in}})$ and $(\tilde{Y}^{n,\mathrm{in}}, \tilde{U}^{n,\mathrm{in}})$, respectively. Since the empirical measures are constructed as in the proof of Theorem \ref{T2.2}, we have
\[\sum_{j=1}^{\infty}\mathbf{m}_j^n u_j^{n,\mathrm{in}} = \sum_{\ell,k=1}^{\infty}m_{\ell,k}^nv_\ell^{n,\mathrm{in}}=\sum_{\ell=1}^\infty m_\ell^n v_\ell^{n,\mathrm{in}}=0,\]
and similarly
\[\sum_{j=1}^{\infty}\mathbf{m}_j^n \tilde{u}_j^{n,\mathrm{in}}=0.\]
We can also rewrite the approximate empirical measures as 
\begin{equation} \label{D-14}
	\mu_t^n=\sum_{j=1}^\infty \mathbf{m}^n_j \delta_{y^n_j(t)}\otimes \delta_{u^n_j(t)} \quad\text{and} \quad \nu_t^n= \sum_{j=1}^\infty \mathbf{m}^n_j\delta_{\tilde{y}^n_j(t)}\otimes \delta_{\tilde{u}^n_j(t)}. 	
\end{equation}
\vspace{.3cm}

\noindent $\bullet$ \textbf{Step B (Derivation of uniform-in-time stability results):} We choose the transport measure $\pi_t\in \mathcal{P}(\bbr^{2d}\times\bbr^{2d})$ of $\mu_t^{n}$ and $\nu_t^{n}$ as 
\begin{equation}\label{D-15}
	\pi_t = \sum_{i=1}^\infty \textbf{m}^n_i \delta_{y^n_i(t)}\otimes \delta_{u^n_i(t)}\otimes \delta_{\tilde{y}^n_i(t)}\otimes \delta_{\tilde{u}^n_i(t)}.
\end{equation}
Now, we use \eqref{D-14} and \eqref{D-15} to find
\begin{align*}
	W_p^p(\mu_t^{n},\nu_t^{n}) &= W_p^p\left(\sum_{i
	=1}^\infty m_i^n \delta_{x_{i}^n(t)}\otimes \delta_{v_{i}^n(t)},\sum_{i=1}^\infty \tilde{m}^n_i \delta_{\tilde{x}^n_{i}(t)}\otimes \delta_{\tilde{v}^n_{i}(t)}\right)\\
		&\leq \int_{\bbr^{4d}}|(x,v)-(\tilde{x},\tilde{v})|_p^p~d\pi_t(x,v,\tilde{x},\tilde{v})\\
	&=\sum_{i=1}^\infty \textbf{m}^n_i\left(|y^n_i(t)-\tilde{y}^n_i(t)|_p^p+|u^n_i(t)-\tilde{u}^n_i(t)|_p^p\right)\\
	&=\|Y^n(t)-\tilde{Y}^n(t)\|_p^p+\|U^n(t)-\tilde{U}^n(t)\|_p^p.
\end{align*}
Then, Theorem \ref{T2.1} and \eqref{D-13} yield
\begin{align*}
	W_p(\mu_t^{n},\nu_t^{n})&\leq\left( \|Y^n(t)-\tilde{Y}^n(t)\|_p^p+\|U^n(t)-\tilde{U}^n(t)\|_p^p\right)^{\frac{1}{p}}\\
	&\leq\|Y^n(t)-\tilde{Y}^n(t)\|_p+\|U^n(t)-\tilde{U}^n(t)\|_p\\
&	\leq C(\|Y^{n,\mathrm{in}}-\tilde{Y}^{n,\mathrm{in}}\|_{p}+ \|U^{n,\mathrm{in}}-\tilde{U}^{n,\mathrm{in}}\|_{p})\\
&\le2^{1-\frac{1}{p}} C(\|Y^{n,\mathrm{in}}-\tilde{Y}^{n,\mathrm{in}}\|_{p}^p+ \|U^{n,\mathrm{in}}-\tilde{U}^{n,\mathrm{in}}\|_{p}^p)^{\frac{1}{p}}\\
&=2^{1-\frac{1}{p}}CW_{p}(\mu_0^n, \nu_0^n),
\end{align*}
where we used the relation:
\[ a+b\leq 2^{1-\frac{1}{r}}(a^r+b^r)^{\frac{1}{r}} \quad \mbox{for $a,b>0$ and $r\geq 1$}. \]
Note that as in the proof of Theorem \ref{T2.2}, we can choose $C>0$ independent of $n$. 
Therefore, we use \eqref{D-10} to get
\begin{equation*}
	W_p(\mu_t,\nu_t)\le W_p(\mu_t^{n},\nu_t^{n}) +\e\leq 2^{1-\frac{1}{p}}CW_{p}(\mu^n_0, \nu^n_0)+\e.
\end{equation*}
Note that Theorem \ref{T2.2} and the triangle inequality give
\[W_p(\mu_0^n, \nu_0^n)\to W_p(\mu_0, \nu_0)\quad \text{as}\quad n\to\infty.\] 
Since $\e>0$ is arbitrary, we conclude the desired result:
\begin{equation*}
	\sup_{t\geq 0}W_p(\mu_t,\nu_t)\leq 2^{1-\frac{1}{p}}CW_{p}(\mu_0, \nu_0).
\end{equation*}

\section{Conclusion}\label{sec:5}
\setcounter{equation}{0} 
In this paper, we have established uniform-in-time stability for the ICS-type model and the KCS-type model in a noncompact setting. To explain the difficulty in the derivation of the desired uniform estimates, we first identified two fundamental obstructions in the classical framework: the velocity diameter of the original ICS model may remain constant in time, and the original KCS model may fail to satisfy uniform-in-time stability in the noncompact spatial setting. These observations illustrate the necessity of introducing a modified communication weight based on a pairwise spatial moment, which is our KCS-type model. To overcome the analytical difficulties, we introduced the SDDI, which enables us to derive quantitative stability estimates uniform in time along the ICS-type model. This uniform stability estimate further led to the uniform-in-time mean-field limit of the corresponding particle model and the corresponding stability theory for the derived KCS-type model. The presented uniform estimates provide a unified framework for analyzing long-time robustness of flocking dynamics in an extended spatial regime. Of course, there are several interesting questions to be answered. For example, our presented results are purely deterministic and continuous in time. Hence, whether the presented results are robust with respect to time-discretization and stochastic perturbations will be interesting questions. We leave these interesting issues for a future work.

\appendix

\vspace{1cm}

\section{Proof of Lemma \ref{L2.1}}\label{App-A}
\setcounter{equation}{0} 
\noindent (1) We use \eqref{B-2} and the index change trick $i\leftrightarrow j$ to get 
\[\sum_{i=1}^\infty m_i\dot{v}_i = \kappa \sum_{i,j=1}^\infty m_im_j \phi(|x_j-x_i|_p)(v_j-v_i)=0.\]
(2) For $q=\infty$, it follows from Proposition \ref{P2.1}. Now we define $G:\bbr^d\to \bbr$ as $G(v):=|v|_p^q$. Then we can write
\begin{equation} \label{A.1}
	\|V(t)\|_{q,p}^q=\sum_{i=1}^\infty m_iG(v_i).
\end{equation}
Since the map $G:~v\mapsto |v|_p^q$ is convex, we have 
\begin{equation} \label{A.2}
	(\nabla G(x)-\nabla G(y))\cdot (x-y) \geq 0,\quad \forall~x,y\in \bbr^d.
\end{equation}
Now, we differentiate \eqref{A.1} to get
\begin{align*}
\begin{aligned}
\frac{d}{dt}\left(\sum_{i=1}^\infty m_iG(v_i)\right) &=\sum_{i=1}^\infty m_i\nabla G(v_i)\cdot \dot{v}_i =\kappa\sum_{i,j=1}^\infty m_i\nabla G(v_i)\cdot m_j\phi(|x_i-x_j|_p)(v_j-v_i)\\
&=-\frac{\kappa}{2} \sum_{i,j=1}^\infty m_im_j\phi(|x_i-x_j|_p)(\nabla G(v_i)-\nabla G(v_j))\cdot(v_i-v_j)\leq 0.
\end{aligned}
\end{align*}
Therefore, $\|V(t)\|_{q,p}^q$ is non-increasing, and so is $\|V(t)\|_{q,p}$. Next, we claim 
\begin{equation} \label{A.3}
	\frac{d}{dt}\|X\|_{q,p}\leq \|V\|_{q,p}.
\end{equation}
Since $\frac{dx_i^k}{dt} = v_i^k,$ we have
\begin{align*}
	\frac{d}{dt}|x_i^k|=v_i^k \cdot\mathrm{sgn}(x_i^k)\leq |v_i^k|.
\end{align*}
Then, we use H\"older's inequality to derive
\begin{align*}
\begin{aligned}
	p|x_i|_p^{p-1}\frac{d}{dt}|x_i|_p &=\frac{d}{dt}|x_i|_p^p \leq p\sum_{k=1}^d |x_i^k|^{p-1}|v_i^k| \\
	&\leq p\left(\sum_{k=1}^d |x_i^k|^p\right)^{\frac{p-1}{p}}\left(\sum_{k=1}^d |v_i^k|^p\right)^{\frac{1}{p}} = p|x_i|_p^{p-1}|v_i|_p.
\end{aligned}
\end{align*}
This yields
\begin{equation} \label{A.4}
\Big| \frac{d}{dt}|x_i|_p \Big| \leq |v_i|_p.
\end{equation}
We again use H\"older's inequality to find 
\begin{align*}
\begin{aligned}
\Big| q\|X\|_{q,p}^{q-1}\frac{d}{dt}\|X\|_{q,p} \Big| &= \Big| \frac{d}{dt}\|X\|_{q,p}^q \Big| = p \Big| \sum_{i=1}^\infty m_i|x_i|_p^{q-1}\frac{d}{dt}|x_i|_p \Big| \\
&\leq q\sum_{i=1}^\infty m_i|x_i|_p^{q-1}|v_i|_p \leq q\left(\sum_{i=1}^\infty m_i |x_i|_q^q\right)^{\frac{q-1}{q}}\left(\sum_{i=1}^\infty m_i|v_i|_p^q\right)^{\frac{1}{q}} \\
& =q\|X\|^{q-1}_{q,p}\|V\|_{q,p}.
\end{aligned}
\end{align*}
We divide both sides by $q\|X\|_{q,p}^{q-1}$, which proves our claim \eqref{A.3}. We combine \eqref{A.3} and $\|V(t)\|_{q,p}\le \|V^{\mathrm{in}}\|_{q,p}$ to get
\begin{equation*}
	\frac{d}{dt}\|X\|_{q,p}\leq \|V^{\mathrm{in}}\|_{q,p}.
\end{equation*}
We integrate both sides of the above relation from $0$ to $t$ to get the desired result.\\\\

\noindent (3) By \eqref{A.4}, we have
\[\frac{d}{dt}\sum_{i=1}^\infty m_i e^{\alpha|x_i|_p}\leq \sum_{i=1}^\infty m_i e^{\alpha |x_i|_p}\alpha |v_i|_p \leq \alpha \cP_\infty \sum_{i=1}^\infty m_i e^{\alpha|x_i|_p}.\]
This leads to the desired estimate:
\[\sum_{i=1}^\infty m_i e^{\alpha |x_i|_p}\leq e^{\alpha \cP_\infty t}\sum_{i=1}^\infty m_i e^{\alpha |x_i^{\mathrm{in}}|_p}.\]

\vspace{1cm}

\section{Proof of Proposition \ref{P2.4}}\label{App-B}
\setcounter{equation}{0} 
Note that 
\begin{equation*}
	\|V(t)\|_{q,p}^q=\sum_{i=1}^\infty m_iG(v_i).
\end{equation*}
Then, we use Lemma \ref{L2.3} to derive
\begin{align} \label{B.1}
	\begin{aligned}
		\frac{d}{dt}\sum_{i=1}^\infty m_i|v_i|_p^q &= -\frac{\kappa}{2} \sum_{i,j=1}^\infty m_im_j\phi(|x_i-x_j|_p)(\nabla G(v_i)-\nabla G(v_j))\cdot(v_i-v_j)\\
		&\leq -\frac{\kappa c_{q}}{2}\sum_{i,j=1}^\infty m_im_j\phi(|x_i-x_j|_p)|v_i-v_j|_p^q\\
		&\leq -\frac{\kappa c_{q}}{2}\phi(2R(t))\sum_{i,j\in \Lambda_t}m_im_j|v_i-v_j|_p^q.
	\end{aligned}
\end{align}
We use 
\[\sum_{i,j\in \Lambda_t} = \sum_{i,j=1}^\infty -2\sum_{i=1}^\infty\sum_{j\in \Lambda_t^c}+\sum_{i,j\in \Lambda_t^c}\]
on \eqref{B.1} to get
\begin{align}
\begin{aligned} \label{B.2}
\frac{d}{dt}\sum_{i=1}^\infty m_i|v_i|_p^q &\leq -\frac{\kappa c_{q}}{2}\phi(2R(t))\sum_{i,j=1}^\infty m_im_j|v_i-v_j|_p^q \\
&\hspace{0.5cm} +\kappa c_{q}\phi(2R(t))\sum_{i=1}^\infty\sum_{j\in \Lambda_t^c}m_i m_j|v_i-v_j|_p^q.
\end{aligned}
\end{align}
Now, since $w\mapsto |w|_p^q$ is convex with 
\[ \sum_{j=1}^\infty m_j=1 \quad \mbox{and} \quad \sum_{j=1}^\infty m_jv_j=0, \]
we use Jensen's inequality to get
\[
	\sum_{i=1}^\infty m_i |v_i|_p^q  = \sum_{i=1}^\infty m_i\left|\sum_{j=1}^\infty  m_j(v_i-v_j)\right|_p^q \leq\sum_{i,j=1}^\infty m_i m_j |v_i-v_j|_p^q.
\]
This implies
\begin{equation} \label{B.3}
	-\frac{\kappa c_{q}}{2}\phi(2R(t))\sum_{i,j=1}^\infty m_im_j|v_i-v_j|_p^q \leq -\frac{\kappa c_{q}}{2}\phi(2R(t))\sum_{i=1}^\infty m_i|v_i|_p^q.
\end{equation}
Next, we use $|v_i-v_j|_p^q \leq 2^{q-1}(|v_i|_p^q+|v_j|_p^q)$ to derive
\begin{align}\label{B.4}
	\begin{aligned}
		& \sum_{i=1}^\infty \sum_{j\in \Lambda_t^c}m_im_j|v_i-v_j|_p^q \\
		&  \hspace{1cm} \leq 2^{q-1}\left(\sum_{j\in \Lambda_t^c}m_j \right)\sum_{i=1}^\infty m_i|v_i|_p^q+2^{q-1}\sum_{j\in \Lambda_t^c}m_j|v_j|_p^q\\
		&  \hspace{1cm} \leq 2^{q-1}\left(\sum_{j\in \Lambda_t^c}m_j\right)\sum_{i=1}^\infty m_i|v_i|_p^q + 2^{q-1}P_\infty^q\sum_{j\in \Lambda_t^c}m_j.
	\end{aligned}
\end{align}
We combine \eqref{B.2}, \eqref{B.3}, and \eqref{B.4} to derive
\begin{align*}
	\frac{d}{dt}\sum_{i=1}^\infty m_i|v_i|_p^q \leq -\frac{\kappa c_{q}}{2}\phi(2R(t))\left(1-2^q\sum_{j\in \Lambda_t^c}m_j\right)\sum_{i=1}^\infty m_i|v_i|_p^q +2^{q-1}c_q\kappa P_\infty^q\phi(2R(t)) \sum_{j\in \Lambda_t^c}m_j.
\end{align*}
We choose $R(t)$ to satisfy
\[\sum_{j\in \Lambda_0^c}m_j\leq \frac{1}{2^{q+1}}.\]
This leads to 
\begin{equation} \label{B.5}
	\frac{d}{dt}\sum_{i=1}^\infty m_i|v_i|_p^q \leq -\frac{\kappa c_{q}}{4}\phi(2R(t))\sum_{i=1}^\infty m_i|v_i|_p^q +2^{q-1}c_q \kappa P_\infty^q \phi(2R(t))\sum_{j\in \Lambda_t^c}m_j.
\end{equation}
Now, we apply Gr\"onwall's lemma for \eqref{B.5} to derive
\begin{align} \label{B.6}
	\begin{aligned}
		\sum_{i=1}^\infty m_i |v_i(t)|_p^q&\leq \exp\left({-\frac{\kappa c_{q}}{4}\int_0^t\phi(2R(s))ds}\right)\|V^\mathrm{in}\|_{q,p}^q\\
		&\quad +2^{q-1}c_q\kappa P_\infty^q\int_0^t \exp\left(-\frac{\kappa c_{q}}{4}\int_s^t\phi(2R(u))du\right)\phi(2R(s))\sum_{j\in \Lambda_s^c}m_j~ds.
	\end{aligned}
\end{align}
Recall 
\begin{equation*}
	\phi(r) = \frac{1}{(1+r^2)^{\beta/2}}, \quad R(t) \simeq (1+t)^\gamma,\quad \sum_{j\in \Lambda_t^c}m_j\lesssim (1+t)^{-\bar{q}(\gamma-1)},
\end{equation*}
which gives
\[\phi(2R(t))\simeq (1+t)^{-\beta\gamma}.\]
First, since $\beta\gamma<1$, we have
\begin{equation}	\label{B.7}
	\exp\left({-\frac{\kappa c_{q}}{4}\int_0^t\phi(2R(s))ds}\right)\lesssim \exp\left(-C(1+t)^{1-\beta\gamma}\right),
\end{equation}
where $C>0$ is a constant. 
For the latter part of \eqref{B.6}, since
\[\int_s^t\phi(2R(u))du \gtrsim\int_s^t (1+u)^{-\beta\gamma}du\geq (t-s)(1+t)^{-\beta\gamma},\]
we have
\begin{align*}
	&\int_0^t \exp\left(-\frac{\kappa c_{q}}{4}\int_s^t\phi(2R(u))du\right)\phi(2R(s))\sum_{j\in \Lambda_s^c}m_j~ds\\
	& \hspace{1cm} \lesssim \int_0^t \exp\left(-c(t-s)(1+t)^{-\beta\gamma}\right)(1+s)^{-\beta\gamma-\bar{q}(\gamma-1)}ds\qquad \\
	& \hspace{1cm} =\int_0^t \exp\left(-cr(1+t)^{-\beta\gamma}\right)(1+t-r)^{-\beta\gamma -\bar{q}(\gamma-1)}dr,
\end{align*}
where $r=t-s$. \newline

Note that
\begin{align} \label{B.8}
	\begin{aligned}
		&\int_{\frac{t}{2}}^t\exp\left(-cr(1+t)^{-\beta\gamma}\right)\underbrace{(1+t-r)^{-\beta\gamma-\bar{q}(\gamma-1)}}_{\leq 1}dr \\
		& \hspace{1cm} \leq \int_{\frac{t}{2}}^\infty  \exp\left(-cr(1+t)^{-\beta\gamma}\right)dr  \lesssim(1+t)^{\beta\gamma}\exp\left(-c'(1+t)^{1-\beta\gamma}\right),
	\end{aligned}
\end{align}
where $c'>0$ is a constant. 
Also, since $1+t-r\geq 1+\frac{t}{2}$ for  $r\in\left[0,\frac{t}{2}\right]$, we have
\begin{align}\label{B.9}
	\begin{aligned}
		&\int_{0}^{\frac{t}{2}}\exp\left(-cr(1+t)^{-\beta\gamma}\right)(1+t-r)^{-\beta\gamma-\bar{q}(\gamma-1)}dr\\
		& \hspace{1cm} \lesssim (1+t)^{-\beta\gamma-\bar{q}(\gamma-1)}\int_{0}^\infty \exp\left(-cr(1+t)^{-\beta\gamma}\right)dr\\
		& \hspace{1cm} \lesssim (1+t)^{-\beta\gamma-\bar{q}(\gamma-1)}(1+t)^{\beta \gamma} = (1+t)^{-\bar{q}(\gamma-1)}.
	\end{aligned}
\end{align}
We combine \eqref{B.6}, \eqref{B.7}, \eqref{B.8}, and \eqref{B.9} to get 
\begin{align*}
	\|V(t)\|_{q,p}^q\lesssim (1+t)^{-\bar{q}(\gamma-1)}.
\end{align*} 
Therefore, we have
\begin{equation*}
	\|V(t)\|_{q,p}\lesssim (1+t)^{-\frac{\bar{q}}{q}(\gamma-1)}.
\end{equation*}

\vspace{1cm}

\section{Proof of Proposition \ref{P3.1}}\label{App-C}
\setcounter{equation}{0} 
\noindent We first define energy functionals $\mathcal{E}_{\pm}$:
\begin{align*}
\mathcal{E}_{\pm}:=	\|V\|_{q,p}\pm\frac{\kappa \tilde{c}_q}{4} \int_{2\|X^{\mathrm{in}}\|_{q,p}}^{2\|X\|_{q,p}}\phi(s)d s.
\end{align*}
By Lemma \ref{L3.1} we know 
\[\mathcal{E}_{\pm}\le	\mathcal{E}_{\pm}^{\mathrm{in}} = \|V^{\mathrm{in}}\|_{q,p}. \]
This implies 
\[\frac{\kappa \tilde{c}_q}{4} \int_{2\|X^{\mathrm{in}}\|_{q,p}}^{2\|X\|_{q,p}}\phi(s)d s\le\|V^{\mathrm{in}}\|_{q,p}, \]
and there exists an $x_{q,p}^{\infty}<\infty$ such that 
\[\|X\|_{q,p}<x_{q,p}^{\infty}.\]
Here, $x_{q,p}^{\infty}$ is the largest positive number such that 
\[\|V^{\mathrm{in}}\|_{q,p}=\frac{\kappa \tilde{c}_q}{4} \int_{2\|X^{\mathrm{in}}\|_{q,p}}^{2x_{q,p}^{\infty}}\phi(s)d s.\]
The velocity alignment follows from $\eqref{SDDI}_2$ and $\|X\|_{q,p}<x_{q,p}^{\infty}$.

\end{document}